\documentclass{article}

\usepackage{fullpage}          
\usepackage{amsfonts}        
\usepackage{amsmath}         
\usepackage{amssymb}         
\usepackage{hyperref}        
\usepackage[numbers]{natbib} 
\usepackage{subcaption}      
\usepackage{graphicx}        
\usepackage{algorithmic}     
\usepackage{multirow}        
\usepackage{booktabs}        
\usepackage{amsthm}

\newcommand{\argmin}[1]{\underset{#1}{\mathrm{argmin}}}
\newcommand{\argmax}[1]{\underset{#1}{\mathrm{argmax}}}

\newcommand{\prox}{\mathbf{prox}}
\newcommand{\proj}{\mathbf{proj}}
\newcommand{\diag}{\mathrm{diag}}

\newcommand{\sign}{\mathbf{sign}}

\newcommand{\mb}{\mathbf}

\newcommand{\vA}{{\mathbf{A}}}
\newcommand{\vU}{{\mathbf{U}}}
\newcommand{\vV}{{\mathbf{V}}}
\newcommand{\vR}{{\mathbf{R}}}
\newcommand{\vb}{{\mathbf{b}}}

\newcommand{\vg}{{\mathbf{g}}}

\newcommand{\vr}{{\mathbf{r}}}

\newcommand{\vu}{{\mathbf{u}}}

\newcommand{\vv}{{\mathbf{v}}}
\newcommand{\vx}{{\mathbf{x}}}
\newcommand{\vy}{{\mathbf{y}}}
\newcommand{\vz}{{\mathbf{z}}}

\newcommand{\PPM}{{\textsc{PPM}}}
\newcommand{\mC}{\mathcal{C}}

\newcommand{\mI}{\mathcal{I}}

\newcommand{\red}[1]{{\color{red}#1}}

\newcommand{\R}{\mathbb{R}}

\newtheorem{theorem}{Theorem}[section]  
\newtheorem{proposition}[theorem]{Proposition}  
\newtheorem{lemma}[theorem]{Lemma}  
\newtheorem{remark}[theorem]{Remark}  

\theoremstyle{definition}
\newtheorem{definition}{Definition}[section]  

\begin{document}

\title{Proximal Flow Inspired Multi-Step Methods}

\author{
    Yushen Huang \thanks{Stony Brook University, Dept. of Computer Science, \href{mailto:yushen.huang@stonybrook.edu}{yushen.huang@stonybrook.edu}} \and
    Yifan Sun \thanks{Stony Brook University, Dept. of Computer Science, \href{mailto:yifan.sun@stonybrook.edu}{yifan.sun@stonybrook.edu}}
}

\maketitle

\begin{abstract}
We investigate a family of approximate multi-step proximal point methods, framed as implicit linear discretizations of gradient flow. The resulting methods are multi-step proximal point methods, with similar computational cost in each update as the proximal point method. We explore several optimization methods where applying an approximate multistep proximal points method results in improved convergence behavior.
We also include convergence analysis for the proposed method in several problem settings: quadratic problems, general problems that are strongly or weakly (non)convex, and accelerated results for alternating projections.
\end{abstract}
\section{Introduction}
In this paper, we consider the following optimization problem:
\begin{equation}\label{eq:obj}
    \min_{\vx \in \mathbb{R}^n} \; F(\vx) = f(\vx) + h(\vx)  
\end{equation}
where $f(\vx)$ is an $L$-smooth function, $h(\vx)$ is a closed, convex but not necessary smooth function, and $F(\vx)$ is bounded below. 
The problem with the following settings has been raised in many applications in data science, machine learning, and image processing~\cite{tibshirani1996regression,yuan2006model,candes2012exact,friedman2008sparse}. 
In this paper, we consider a family of multi-step proximal point updates.
The algorithm is a generalization of the proximal point method (PPM)~\cite{moreau1965proximite} 
where we use a linear combination of the previous $\tau$ steps in each iteration, as follows:
\begin{equation}
\tilde \vx^{(k)} = \sum_{i=1}^\tau \xi_i \vx^{(k-\tau+i)}, \qquad 
    \vx^{(k+1)}  = 
    \mathcal F( \tilde{\vx}^{(k)}).
    \label{eq:main}
\end{equation}
Here, $\mathcal F$ is an approximate proximal point step. When $\tau = 1$ and $\xi_1 = 1$,  \eqref{eq:main} reduces to the ``vanilla" approximate proximal point method, of which there are many works  \cite{moreau1965proximite,asi2019stochastic,nesterov2021inexact,asi2020minibatch}.
In this paper, we investigate improvements garnered by higher order $\tau > 1$. Note that 
there is very little overhead in increasing $\tau$, as the averaging is done in an online manner. 
Moreover,  while tuning $\tau$ and $\xi_i$ to a wider range of values would often result in a more general class of methods, this paper focuses on choices of $\tau$ and $\xi_i$ are generalizable across variations of problems, to avoid cumbersome parameter tuning. 
%
%
This is often possible, for example, when $\xi_i$ and $\tau$ are inspired by \emph{implicit discretization methods of gradient flow}, which have been heavily studied \cite{su2014differential,shi2019acceleration,zhang2018direct,romero2020finite}. 
%
%
%
Specifically, we study \eqref{eq:main} where $\xi_i$ are inspired by the \emph{backwards differential formula} \cite{ascher1998computer}. 

%
%
We give convergence result of such extensions, under the presence of a $\gamma$ - contractive approximate proximal point method. 
We also demonstrate numerically that these methods give improvements in several important applications, such as proximal gradient in compressed sensing with both convex and nonconvex penalties,  alternating projections over linear subspaces, and alternating minimization for matrix factorization. 
\subsection{Related Work}

\paragraph{Proximal point method.} 


The proximal point method, originally introduced in~\cite{moreau1965proximite}, involves minimizing the following subproblem at each iteration:
 \[
\vx^{(k+1)} := \argmin{\vu}\; F(\vu) + \frac{1}{2\beta} \Vert \vu - \vx^{(k)} \Vert^2.
\]
In cases where the objective function $F$ is convex, this subproblem exhibits a larger strong convexity parameter, thereby enabling faster numerical methods. In practical applications, an approximate solution to the subproblem is obtained at each iteration. This is  achieved using methods such as stochastic projected subgradient~\citep{davis2019proximally, asi2019stochastic}, prox-linear algorithm~\citep{drusvyatskiy2018error}, and the catalyst generic acceleration schema~\citep{lin2015universal}.
Furthermore, the proximal point method can be generalized~\citep{nesterov2021inexact} by altering the penalty norm to 
$\Vert \cdot \Vert_2^{p+1}$, where $p\geq 1$.  Such generalizations often result in faster convergence rates compared to the standard proximal point methods.

\paragraph{Methods inspired by dynamical systems.}  



Several works have leveraged the insight that minimizing a function is equivalent to finding the stationary point of a dynamical system. For instance, \cite{su2014differential} analyzed the Nesterov accelerated gradient descent method as the discretization of a second-order ordinary differential equation. Similarly, \cite{shi2019acceleration} examined a method formed by employing a higher-order symplectic discretization of a related differential equation, as detailed in \cite{shi2021understanding}, to achieve acceleration. Others have pursued a similar approach using Runge-Kutta explicit discretizations, as demonstrated in \cite{zhang2018direct}. Some researchers have explored optimization methods by discretizing alternative flows, such as rescaled gradient flow \cite{wilson2019accelerating}.
A noteworthy example of optimization improvements inspired by discretization is the extragradient method \cite{korpelevich1976extragradient}, widely applied in min-max optimization problems \citep{du1995minimax} and variational inequality problems \citep{facchinei2003finite}. The connection between dynamical systems and optimization methods also appears in other works \citep{schropp2000dynamical, wilson2019accelerating, krichene2015accelerated, orecchia2018accelerated, sundaramoorthi2018variational}.



\subsection{Contributions}
In this paper, we do the following.
\begin{itemize}
    \item We propose a dynamical system that, when discretized either explicitly or implicitly, leads to gradient descent, the proximal gradient method, or the proximal point method. 

    \item We propose a higher order implicit discretization scheme, based on the backward differentiation formulas, which are computationally efficient extensions of existing methods, and offer stabelizing improvements to certain methods that come from unstable discretizations of gradient flow.

    \item We give convergence analyses on smooth quadratic problems, as well as convex, strongly convex, and nonconvex smooth problems.

    \item We apply our methods to several problems, such as proximal gradient over  nonconvex sparse regularization, alternating minimization and alternating projections. Numerically, we see that in these scenarios, multi-step methods improve noticeably over their single-step  versions. 

\item Finally, we show that in the case of alternating projections, the coefficients can be tuned to accelerate convergence.

\end{itemize}

The rest of this paper is organized as follows.
 Section \ref{sec:dynsystem} introduces the proximal flow ODE and shows that the solution is well-defined. 
Section \ref{sec:implicit_discr} introduces the BDF scheme as the multi-step implicit discretization.
 Sections \ref{sec:quad_conv}  and \ref{sec:conv_general} gives convergence analyses for quadratic and general functions, respectively.
Section \ref{sec:applications} applies the BDF discretization of the prox-flow ODE to four applications, and shows improved numerical behavior. 
Section \ref{sec:altproj} explores acceleration in alternating projection instances when $\xi$ is tuned, and section \ref{sec:conclusion} concludes the work.

\section{Dynamical systems characterization}

\label{sec:dynsystem}
In this paper, we assume all vector norms are 2-norms, e.g. $\|\cdot\| = \|\cdot\|_2$.

\subsection{Function classes}

\begin{definition}
A function $F:\mathbb{R}^n\to \mathbb{R}$ is \emph{smooth} if for all $\vx$, $\nabla F(\vx)$ exists. 
It is additionally \emph{$L$-smooth} if its gradient is $L$-Lipschitz:
\[
\|\nabla F(\vx) - \nabla F(\vy)\|\leq L\|\vx-\vy\|, \quad \forall \vx,\vy.
\]
\end{definition}
We now give a (relaxed) definition of convexity.
\begin{definition}
 $F$ is 
 \emph{$\mu$-convex at $\vx$}, if
\[
F(\vy) \geq F(\vx) + \langle \nabla F(\vx),\vy-\vx\rangle + \frac{\mu}{2}\|\vy-\vx\|^2. 
\]
for any $\vy$. 
If $F$ is $\mu$-convex for all points $\vx$, we say that $F$ is $\mu$-convex.
\end{definition}
Specifically, $\mu \geq 0$ implies $F$ is convex, and if $\mu < 0$ then the function may be nonconvex. 
Note that if $F$ is $L$-smooth, it is also $-L$-convex. 
However, the condition of  $\mu$-convex with a negative $\mu$ is more general; for example, the function $F(\vx) = \vx^4-\vx^2$  is not $L$-smooth nor convex, but is weakly convex ($\mu = -2$).

\subsection{The proximal flow}
%
Consider \eqref{eq:obj} where $f(\vx)$ is  $L$-smooth and $h(\vx)$ is a convex function. 
The local minimum of $F(\vx) = f(\vx) + h(\vx)$ is a stationary point of the following differential equation (DE), which we term the \emph{proximal flow}
%
%
%
\begin{equation}\label{eq:prox_flow}
    \dot{\vx}(t) = \lim_{\beta \to 0^+} \frac{\prox_{\beta F}(\vx(t)) - \vx(t)}{\beta}
\end{equation}
where
\[
\prox_{\beta F}(\vx) = \argmin{\vz}\,F(\vz) + \frac{1}{2\beta}\|\vx-\vz\|_2^2
\]
is the \emph{proximal operator} of the function $F:\R^n\to\R$ at point $\vx\in \R^n$. The parameter $\beta > 0$ is often used as a step length control.

\begin{proposition}\label{prop:proxflow_exists}
\eqref{eq:prox_flow} is well-defined; that is, the limit of its right-hand-side always exists and is attained. 
\end{proposition}

\begin{proof}
We first show that for the 1-D parametrized function
\[
g(\beta;\vv) = \vv^T(\prox_{\beta F}(\vx) -\vx),
\]
the directional derivatives 
\begin{equation}
\lim_{\beta\to 0^+} \frac{g(\beta;\vv)-g(0;\vv) }{\beta} = \lim_{\beta\to 0^+} \frac{\vv^T(\prox_{\beta F}(\vx) -\vx) }{\beta} 
\label{eq:prox_exists_helper1}
\end{equation}
exist, for all $\vv$. To see this is true, note that since $F$ is a composition of an $L$-smooth and convex function, it is locally Lipschitz; therefore, $g$ is locally Lipschitz w.r.t. $\beta$. Therefore, either $g(\beta;\vv)$ is differentiable at $\beta = 0 $ (and thus $g(\beta;\vv)$ is continous at $\beta = 0$), or there exists a small neighborhood where, for all $|\beta|\leq \epsilon$, $\beta \neq 0$, $g(\beta;\vv)$ is differentiable. Therefore, the limit \eqref{eq:prox_exists_helper1} exists, for all $\vv$. 
Therefore, the partial derivative $\frac{\partial}{\partial \beta} \prox_{\beta F}(\vx)$ exists, and we may apply l'Hopital's rule and conclude that the limit in \eqref{eq:prox_flow} exists.

\end{proof}

This DE has a close connection with the Moreau envelope of $F$
\[
F_{\beta}(\vx) = \inf_{\vu}\; F(\vu) + \frac{1}{2\beta}\|\vx-\vu\|^2, \qquad 
\nabla F_{\beta}(\vx) = \frac{1}{\beta} \left ( \vx - \prox_{\beta F} (\vx) \right).
\]
Note that this gradient exists and is unique for all  $0<\beta < \frac{L}{2}$.
Also, when $F$ is smooth, \eqref{eq:prox_flow} reduces to vanilla gradient flow
\begin{equation}
\dot \vx(t) = -\nabla F(\vx(t)).
\label{eq:gradient_flow}
\end{equation}

\subsection{Existence and uniqueness} 
To show the existence and uniqueness of \eqref{eq:prox_flow}, we will first show that its right-hand-side always exists, and that it is a special instance of subgradient flow.

\begin{lemma} 
\label{lem:exist_unique}
Consider $F = f + h$ where $f$ is $L$-smooth and $h$ is convex. Then  the right-hand-side of \eqref{eq:prox_flow} always exists, and satisfies
\begin{equation}
\dot \vx(t) \in - \partial F(\vx(t))
\label{eq:subgradflow}
\end{equation}
where $\partial F(\vx(t))$ is the Clarke subdifferential of $F$ at $\vx(t)$ \citep{clarke1990optimization}
\begin{equation}
\partial F(\vx):=\{ \vg \in \R^n : F^D(\vx;\vv) \geq \langle \vv,\vg\rangle,\; \forall \vv \}
\label{eq:clarke-def1}
\end{equation}
with 
\[
F^D(\vx,\vv) := \limsup_{\vy\to \vx,\rho\to 0} \frac{F(\vy+\rho \vv)-F(\vy)}{\rho}.
\]
\end{lemma}

A common alternative representation of the Clarke subdifferential set is the convex hull of all limiting gradients, when $F(\vx)$ is locally convex at $\vx$.
\begin{equation}
\partial F(\vx)=\mathbf{conv}\left(\{\vg : \exists \vu_i  \to \vx, \nabla F(\vu_i)\to \vg\}\right).
\label{eq:clarke-def2}
\end{equation}
However, note that \eqref{eq:clarke-def1} and \eqref{eq:clarke-def2} are not equivalent when $F$ is locally concave at $\vx$. (For example, if $F(\vx) = -|\vx|$, then \eqref{eq:clarke-def2} gives $\partial F(\vx) = [-1,1]$, although the support function doesn't exist at $\vx = 0$.)

\begin{proof}[Proof of Lemma \ref{lem:exist_unique}]
 If $\nabla F(\vx)$ exists, then it is trivially true that 
\[
\lim_{\beta\to 0^+}\frac{\prox_{\beta F}(\vx)-\vx}{\beta} = -\nabla F(\vx) \in -\partial F(\vx).
\]
 Let us consider only when $\nabla F(\vx)$ does not exist.
 Taking $\vx_\beta = \prox_{\beta F}(\vx)$, then 
the right-hand-side of \eqref{eq:prox_flow} is simply 
\[
g:=\frac{\vx_\beta-\vx}{\beta}\in -\partial F(\vx_\beta)
\]
where the inclusion is by optimality of the prox operator. Therefore, 
\[
\vv^Tg \in -\vv^T\partial F(\vx_\beta).
\]
Next, for a fixed $\vy$ and $\vv$, the function $\phi(\rho) = F(\vy + \rho \vv)$ is 1-D (for $\rho \geq 0$), locally Lipschitz  (because $F$ is locally Lipschitz)  and thus all its nonsmooth points are isolated. Therefore, the limit 
\[
\lim_{\rho\to 0^+} \frac{\phi(\rho)-\phi(0)}{\rho} =: \phi'(0^+) 
\]
exists, and moreover, $\phi'(0^+)\in  \vv^T\partial F(\vy)$. 
Therefore,
\[
F^D(\vx;\vv) =  \limsup_{\vy\to \vx} \{\alpha : \alpha \in \vv^T\partial F(\vy)\}
= \sup  \{\alpha : \alpha \in \vv^T\partial F(\vx)\} \geq -\vv^Tg.
\]
Therefore, $-g\in \partial F(\vx)$.

 \end{proof}
We now consider whether  \eqref{eq:prox_flow} always exists and is unique. Specifically, system \eqref{eq:prox_flow} is a differential inclusion. Solutions of such systems exist when $\partial F$ is an upper hemicontinuous map, and is unique when it satisfies a one-sided Lipschitz condition.

\begin{lemma} 
For $F = f+g$ where $f$ is $L$-smooth and $g$ is convex, then the negative of the Clarke subdifferential \citep{clarke1990optimization}  \eqref{eq:clarke-def1}
is always a closed and convex set, and is upper hemicontinuous
\[
\forall \vx, \; \exists \epsilon > 0,\;  \forall \vu:\|\vu-\vx\|\leq \epsilon, \quad  \partial F(\vu)\subseteq \partial F(\vx),
\]
and satisfies the one-sided Lipschitz condition, e.g.
\begin{equation}
 \langle g_x - g_y, \vx - \vy \rangle \leq L \Vert \vx - \vy \Vert^2,
\label{eq:onesidedL}
\end{equation}
for all $g_x\in -\partial F(\vx)$, $g_y\in -\partial F(\vy)$.
\end{lemma}

\begin{proof}
By construction, $\partial F(\vx)$ is always closed and convex. 
Note that $-\partial F(\vx)$ is upper hemicontinuous if $\partial F(\vx)$ is upper hemicontinuous. 
Then, since
\[
\lim_{\vu\to \vx} \limsup_{\vy\to \vu,\rho\to 0}   \frac{F(\vy+\rho \vv)-F(\vy)}{\rho} \leq  \limsup_{\vy\to \vx,\rho\to 0} \frac{F(\vy+\rho \vv)-F(\vy)}{\rho} 
\]
upper hemicontinuity of the Clarke subdifferential immediately follows.
Now we consider if \eqref{eq:prox_flow} is one-sided Lipschitz. Recall that since $f$ is $L$-smooth, it simply follows   that
\[
|\langle \nabla f(\vx)-\nabla f(\vy),\vx-\vy\rangle| =
|\langle g_x-g_y,\vx-\vy\rangle| \leq L \|\vy-\vx\|^2.
\]
Next, since $h$ is convex and $h_x,h_y$ is the subgradient for $h(\vx),h(\vy)$ respectively, 
\[
\langle h_{x}-h_{y}, \vx-\vy\rangle \geq 0 \iff 
\langle g_x-g_y, \vx-\vy\rangle \leq 0.
\]
Therefore, $F = f+h$ satisfies \eqref{eq:onesidedL}.
\end{proof}
\begin{theorem}
   For $F = f+h$ where $f$ is $L$-smooth and $h$ is convex, the proximal flow defined by \eqref{eq:prox_flow} always admits an existing and unique solution $\vx(t)$.
\end{theorem}
\begin{proof}
  Since \eqref{eq:prox_flow} is an instance of subgradient flow \eqref{eq:subgradflow}, and since the Clarke subdifferential is  upper hemicontinuous, then \eqref{eq:prox_flow} always admits an existing solution.
  Additionally, $-\partial F(\vx)$ satisfies the one-sided Lipschitz condition, and therefore the solution is unique.
\end{proof}

\subsection{Optimality}
Next, we show that the stationary point of the proximal flow (e.g. $\lim_{t \to \infty }\vx(t)$) is equivalent to finding the stationary point of $F(\vx)$.

\begin{theorem}\label{the:equiv}
Consider the proximal flow in \eqref{eq:prox_flow} with $\vx(0) = \vx_0$.
Then a stationary point of the ODE   \eqref{eq:prox_flow} is a stationary  point of $F(\vx)$.
\end{theorem}
\begin{proof}
Recall that if $\dot\vx(t) = 0$, then equivalently
\[
 \lim_{\beta \to 0} \frac{\prox_{\beta F}(\vx) - \vx}{\beta} = 0 
\]
which is equivalent to 
\[
 0 = \lim_{\beta \to 0} \nabla F_{\beta}(\vx) \in \lim_{\beta \to 0} \partial F(\prox_{\beta F}(\vx))  = \partial F(\vx).
\]
Hence $0$ is a subgradient of $\partial F(\vx)$, if and only if $\vx$ is a stationary point of $F$.
\end{proof}
%

\subsection{Examples}
To give a more intuitive understanding of the proximal flow, we give the several examples of $F(\vx)$ and demostrate how to calculate its corresponding proximal flow.

\paragraph{Example: Smooth minimization.} Suppose $F(x) = f(x)$ is differentiable everywhere. Then, 
\[
\lim_{\beta \to 0} \frac{\prox_{\beta f}(\vx) - \vx}{\beta} =  -\nabla f(\vx)
\]
and the proximal flow reduces to gradient flow \eqref{eq:gradient_flow}.

\paragraph{Example: Smooth + nonsmooth.} 
Now suppose that $f$ is not only smooth, but is second-order differentiable and
where $h$ is nonsmooth but convex. Then, in the limit of $\beta \to 0$, we may use the Taylor approximation for $\vu \approx \vx$
\[
f(\vu) \approx   \nabla f(\vx)^T(\vu-\vx) 
\]
and therefore, in this limit,
\[
 \prox_{\beta f}(\vx) 
 =\argmin{\vu} \; \beta \nabla f(\vx)^T(\vu-\vx) + \beta h(\vu) + \frac{1}{2}\|\vx-\vu\|_2^2
 = \prox_{\beta h}(\vx-\beta \nabla f(\vx)).
\]
In other words, in the limit of $\beta \to 0$, the proximal gradient method also reduces to the proximal point flow.

\paragraph{Example: Shrinkage.} Let us  consider $h(\vx) = \|\vx\|_1$. Then 
\[
 \prox_h (\vx)_i =  \begin{cases}
    x_i-\sign(x_i)\beta, & |x_i| > \beta,\\
    0 & \text{ else,}
\end{cases} 
\]
and therefore, in the limit $\beta\to 0$, elementwise,
\[
\dot x_i(t) =
  \begin{cases}
    -\sign(x_i), & |x_i| \neq 0,\\
    0 & \text{ else.}
\end{cases}
\]
Note that in this limit, $\dot x_i$ is not continuous (although $x_i$ is still continuous).
And, for all convex nonsmooth $h$,  by using the proximal operator and not the subgradient, we ensure that $\dot\vx(t)$ is always attaining a unique value.


 \paragraph{Example: LSP.} Let us now consider a nonconvex loss  $F(\vx) = \sum_i \log(1+\theta^{-1}|x_i|)$. 
 Note that if $n = 1$, then 
 \[
 F(\vx) = \underbrace{\log\left(1+\frac{|\vx|}{\theta}\right) - \frac{|\vx|}{\theta}}_{g(\vx)}+ \underbrace{\frac{|\vx|}{\theta}}_{h(\vx)}
 \]
where $g$ is $L=\theta^{-2}$ smooth and $h$ is convex. Therefore, \eqref{eq:prox_flow} is well defined for this choice of function.
 Additionally, this proximal flow  may also be computed element-wise:
  \begin{eqnarray*}
\prox_{\beta F} (\vx)_i &=& \argmin{u} ~  \log\left(1+\frac{|u|}{\theta}\right) + \frac{(u-x_i)^2 }{2\beta}\\
&=& 
\begin{cases}
    \frac{\sqrt{(x_i+\theta)^2 -4\theta \beta}-\theta+x_i}{2}\cdot \sign(x_i), & x_i-\theta\geq \sqrt{(x_i+\theta)^2-4\theta \beta},\\
    0 & \text{else.}
\end{cases}
\end{eqnarray*}
 For each $i$, if  $x_i\neq 0$,  then 
 \[
 \dot x_i = -\frac{\partial h(\vx)}{\partial x_i} = -\frac{\sign(x_i)}{b+|x_i|}.
 \]
 Now consider $x_i = 0$. Then 
 \[
\prox_{\beta F} (\vx)_i = \argmin{u} ~ \beta \log\left(1+\frac{|u|}{b}\right) + u^2 
 \]
 which is monotonically increasing in $u$
 for any value of $b > 0$, $\beta > 0$, and is minimized at  $u = 0$. Therefore,
\begin{equation}
\dot x_i =  \frac{\prox_{\beta F}(\vx)_i-\vx_i}{\beta} = \begin{cases}
   -\frac{\sign(x_i)}{b+|x_i|}, & x_i \neq 0,\\
    0 & \text{ else.}
\end{cases}
\label{eq:lspflow}
\end{equation}
Again, this is not a typical DE because $\dot x_i$ is not continuous in $t$; however, $x_i(t)$ \emph{is} continuous in $t$ (and exists and is unique everywhere).
 
\section{Implicit discretization of proximal flow}
\label{sec:implicit_discr}
\subsection{Proximal methods}

%
%
%

At this point, we have now observed three different possible discretizations of the right hand side of \eqref{eq:prox_flow}; that is, we may update as 
\[
\frac{\vx^{(k+1)}-\tilde \vx^{(k)}}{\beta} = \mathcal F(\tilde \vx^{(k)})
\]
where for $\tilde \vx^{(k)} = \vx^{(k)}$, and the choice of $\mathcal F(\vx)$ leads to three well-known methods
\begin{align*}
-\nabla F(\vx), && \text{(Gradient descent)}\\
\frac{\prox_{\beta h}(\vx - \beta \nabla f(\vx))-\vx}{\beta},&& \text{(Proximal gradient method)} \\
\frac{\prox_{\beta F}(\vx) - \vx}{\beta}, && \text{(Proximal point method).}
\end{align*}
We now explore higher order (multi-step) schemes by  simply by modifying $\tilde \vx^{(k)}$:
\begin{equation}
\tilde \vx^{(k)} = \sum_{i=1}^\tau \xi_i \vx^{(k-\tau+i)},
\label{eq:bdf-mixing}
\end{equation}
and at each iteration, computing
\begin{equation}
\vx^{(k+1)}  \approx \argmin{\vx} \; F(\vx) + \frac{1}{2\beta} \|\vx-\tilde {\vx}^{(k)}\|^2.
\label{eq:approximate-prox}
\end{equation}
Specifically, the proximal point method with this choice of $\xi$ specified in table \ref{tab:bdf-const}.
While there exists a variety of methods for producing such discretizations, we focus on the \emph{backwards differential formula}, in which the coefficients for $\xi_i$ are chosen so that the function approximates the slope of the  $\tau$-order Lagrange interpolating polynomial, which explicitly dictates the constants $\xi_i$ as presented in 
Table \ref{tab:bdf-const}.

\subsection{Truncation error and stability}

\paragraph{Truncation error.}
For a multi-step method, the \emph{local truncation error} of an iterative process that discretizes
$\dot{\mathbf{x}}(t) = -\mathcal{F}(\mathbf{x}(t))$,
is  defined as
\begin{equation}
\epsilon_k(\alpha) =  \frac{1}{\alpha} \left ( 
    \sum_{i=0}^{\tau} \xi_i \vx^{(k-\tau+i)} - \alpha  \mathcal F( \vx^{(k)})\right).
\end{equation}
If $\lim_{\alpha \to 0} \epsilon(\alpha) = 0$, the iterative update is considered consistent. Moreover, if the truncation error $\epsilon(\alpha)$ behaves as $\Theta(\alpha^{\tau})$, then the iterative update $\mathcal{A}$ is said to exhibit a truncation error of order $\tau$.
It is well-known \cite{ascher1998computer} that the truncation error of a $\tau$-step BDF method is $O(\alpha^\tau)$.
In other words, the higher order the method, the more likely the iterates $\vx^{(k)}$ hug the continuous trajectory $\vx(t)$.

\paragraph{Stability and step size.}
The region of stability in a discretization method gives the permitted step size ranges such that the method still converges. 
Determining the appropriate step size for a BDF based on the region of stability is less straightforward. On one hand, the entire negative real line is in the stability region for all BDF schemes, and indeed, in many of the problems we discuss, the gradient field is such that $\dot \vx(t) = A\vx(t)$ can be well-modeled by $A$ with purely real eigenvalues. However, because in practice our implicit steps are approximate, instabilities indeed arise when $\tau$ is high. 
We address this  case by case  in our convergence analysis.
\begin{table}[]
\small
    \centering
\begin{tabular}{|c|c|cccc|}
\hline&&&&&\\ [-2ex]
     &  $\bar \xi$ & $\xi_1$ & $\xi_2$ & $\xi_3$ & $\xi_4$  \\ \hline
    BDF1 &1 &  1&&& \\
    BDF2 & 2/3 & -1/3 &  4/3 &&\\
    BDF3 & 6/11 & 2/11 & -9/11 &18/11 &\\
    BDF4 & 12/25& -3/25 & 16/25 & -36/25 & 48/25\\\hline
\end{tabular}
    \caption{Summary of constants for BDF methods. }
    \label{tab:bdf-const}
\end{table}

\section{Quadratic convergence analyses}
\label{sec:quad_conv}
We now present the convergence analyses of the BDF method under  quadratic problems.

\subsection{Approximate proximal methods}
First, we establish our metric of approximation in computing the proximal point step, which is necessary for all remaining sections.
In general, a full proximal point step $\prox_{\beta F}$ is not much computationally cheaper than fully minimizing $F$; therefore, in practice, only approximation methods are reasonable to consider. 
\begin{definition}
Denote the exact solution to the proximal operation   as
\[
\vx_*^{(k+1)}:=  \argmin{\vx} \;  \underbrace{F(\vx) + \frac{1}{2\beta} \|\vx-\tilde {\vx}^{(k)}\|_2^2}_{=:\tilde F_\beta(\vx) }.
\]
    We say that an approximate solution $\vx^{(k +1)}$  is \emph{$\gamma$-contractive} if it satisfies 
\begin{equation}
    \Vert \vx^{(k +1)} - \vx_*^{(k+1)} \Vert_2  \leq \gamma     \Vert \sum_{i=1}^{\tau}\xi_i \vx^{(k-\tau +i)} - \vx_*^{(k+1)} \Vert_2\label{eq:approx_contract_assp}.
\end{equation}
\end{definition}
There are many ways of achieving a $\gamma$-contractive solution, such as using $m$ iterations of a gradient or proximal gradient method, for small $m$.

\paragraph{Example.}
Suppose $F= f +h$ and $f$ is $L$-smooth and $\prox_{h}$ is easy to compute.  Then computing $m$ steps of proximal gradient descent with step size $\alpha = \beta /(\beta L+1)$ over $\tilde{F}_\beta(\vx)$ 
\[
\vx^{(k,i+1)} = \prox_{\beta h}\left(\vx^{(k,i)} - \alpha \nabla F(\vx^{(k,i)})-\frac{\alpha}{\beta}(\vx^{(k,i)}-\tilde\vx^{(k)})\right), 
\]
with
\[
\vx^{(k,0)} = \vx_{*}^{(k)}, \quad \vx^{(k+1)} = \vx^{(k,m)}
\]
satisfies \eqref{eq:approx_contract_assp} with $\gamma = (1-\frac{1}{(L\beta+1)})^m$.


\subsection{Convergence over quadratic functions}
We now consider the error analysis of   BDF over quadratic functions
\[
f(\vx) = \frac{1}{2}\vx^T Q \vx , \qquad h(\vx) = 0.
\]
Here $Q$ is a positive definite matrix with eigenvalues bounded $\mu \leq \lambda_i \leq L$.
We define the radius of convergence of a sequence $\vx^{(k)}$ as $\rho$ where $0\leq \rho\leq 1$ and
\begin{equation}
\|\vx^{(k)}-\vx^*\|\leq \rho^k\|\vx^{(0)}-\vx^*\|.
\label{eq:conv-linear-radius}
\end{equation}
Then, since $\tilde F_\beta(\vx)$ is $\mu=\beta^{-1}$ strongly convex and $L+\beta^{-1}$ smooth, then using an optimal step size of $\frac{\beta}{2+\beta L}$ yields 
\[
 \Vert \vx^{(k +1)} - \vx_*^{(k+1)} \Vert  \leq \underbrace{\left(1-\frac{1}{\beta L  +1}\right)^m}_{\gamma} \Vert \sum_{i=1}^{\tau}\xi_i \vx^{(k-\tau +i)} - \vx_*^{(k+1)} \Vert.
 \]
Then, the algorithm dictated by \eqref{eq:bdf-mixing} and \eqref{eq:approximate-prox} can be written precisely as  
\[
\vx^{(k+\tau + 1)} = A^{m} \vx^{(k+\tau)}  + \frac{\alpha}{\beta}\sum_{j=1}^m A^{j-1} \sum_{i=1}^\tau \xi_i \vx^{(k+i)}, 
\qquad
A =\left(1-\frac{\alpha}{\beta}\right) I-\alpha Q.
\]
Defining also 
$B:=\frac{\beta}{\alpha} \sum_{j=1}^m A^{j-1}$
we may construct a first-order system
\[
\underbrace{\begin{bmatrix}
\vx^{(k+\tau+1)}\\
\vx^{(k+\tau)}\\
\vx^{(k+\tau-1)}\\
\vdots\\
\vx^{(k+2)}\\
\end{bmatrix}}_{\vz^{(k+1)}}
 \hspace{-1ex} = \hspace{-1ex}
 \setlength{\arraycolsep}{1mm}
\underbrace{\begin{bmatrix}
A^m +   \xi_{\tau}B & \xi_{\tau-1} B& \cdots &   \xi_1B\\
I&0&\cdots & 0  \\
0&I&\cdots & 0 \\
\vdots & \vdots & \ddots & \vdots  \\
0&0&\cdots  & 0 \\
\end{bmatrix}}_{=: M}
\hspace{-1ex}
\underbrace{\begin{bmatrix}
\vx^{(k+\tau)}\\
\vx^{(k+\tau-1)}\\
\vx^{(k+\tau-2)}\\
\vdots\\
\vx^{(k+1)}\\
\end{bmatrix}}_{\vz^{(k)}}
\]
and defining $\vz^* = \textbf{vec} (\vx^*,\vx^*,...,\vx^*)$
we see that it is the eigenvalues of $M$ that define the method radius of convergence
\[
(\vz^{(k)}-\vz^*) = M^k (\vz^{(0)}-\vz^*) .
\]
Unfortunately, $\|M\|_2$ is bounded below by 1, so we invoke Gelfand's rule \citep{gelfand1941normierte} to produce an asymptotic bound
\begin{equation}
\|\vz^{(k)}-\vz^*\| \lesssim (\mathrm{Re}(\lambda_{\max}(M)))^k\|\vz^{(0)}-\vz^*\|.
\label{eq:contract-bdf}
\end{equation}
That is, although $M$ is not symmetric, \eqref{eq:contract-bdf} holds asymptotically.

In Tables \ref{tab:quad_sensitivity} and \ref{tab:quad_optimal}, we give the limits on step size $\alpha$ and on optimal $\rho$ by finding values in which $\lambda_{\max}(M) < 1$, guaranteeing stability. Alongside, Figure \ref{fig:multistep_poly_stability} shows $\rho$ over a continuation of $\beta$, over changing $m$ and $\alpha$. There is a clear correlation with greater stability and faster convergence with increased $\beta$ (implicit step size), as expected. The reliance on $m$ is interestingly unexpected; more $m$ leads to better solving of the inner prox step, but gives little effect on stability and can even give a negative effect on convergence. This suggests that though implicit methods are thought of as too expensive, very approximate versions are not only practical, they are very close to optimal. 

\begin{table}[ht]

\begin{subtable}[t]{.45\textwidth}
    \centering
    \begin{tabular}[t]{|l|cc|}
    \hline
&$L = 2$&$L = 10$\\\hline
     \PPM (1)&0.667 & 0.182  \\
     \PPM (10)&0.952 & 0.198 \\
     BDF2 (1)&0.665 & 0.181   \\
     BDF2 (10)& 0.940 & 0.197 \\
     BDF3 (1)&0.608 & 0.178  \\
     BDF3 (10)&0.940 & 0.197 \\
     \hline
\end{tabular}
    \caption{\textbf{Sensitivity of parameters.} Upper bound on best step size to avoid divergence $\alpha$ for PPM and BDFs).  \PPM$(\beta)$ means proximal point, with step size $\beta$. (Bigger is better.) Here, we pick $m = 4$, and notice very little deviation for $m > 4$. Here, $\mu = 1$.}
    \label{tab:quad_sensitivity}
\end{subtable}'
\hfill
\begin{subtable}[t]{.45\textwidth}
\begin{tabular}[t]{|l|cc|}
    \hline
&$L = 2$&$L = 10$\\\hline
     \PPM (4,1)& 0.500 & 0.596 \\
     \PPM (20,1)& 0.500 & 0.500 \\
     \PPM (4, 10)& 0.0935 & 0.466 \\
     \PPM (20,10) & 0.0909 & 0.100 \\
     BDF2 (4,1)&0.326 & 0.282    \\
     BDF2 (20,1)&0.303 & 0.211    \\
     BDF2 (4,10)&0.059 & 0.423  \\
     BDF2 (20,10)&0.024 & 0.024  \\
     BDF3 (4,1)& 0.377 & 0.451 \\
     BDF3 (20,1)&0.377 & 0.306  \\
     BDF3 (4,10)&0.197 & 0.459   \\
     BDF3 (20,10)&0.197 & 0.165   \\\hline
\end{tabular}
    \caption{\textbf{Optimal $\rho$.} Minimum $\rho$ given best best step size, $\alpha$ for PPM and BDFs). PPM$(m,\beta)$ means proximal point, with step size $\beta$ over $m$ iterations. (Smaller is better.) Here, $\mu = 1$.}
    \label{tab:quad_optimal}
\end{subtable}
\end{table}

\begin{table}[ht]
    \centering
    
\end{table}
\begin{figure*}
    \centering
    \includegraphics[width=0.8\textwidth]{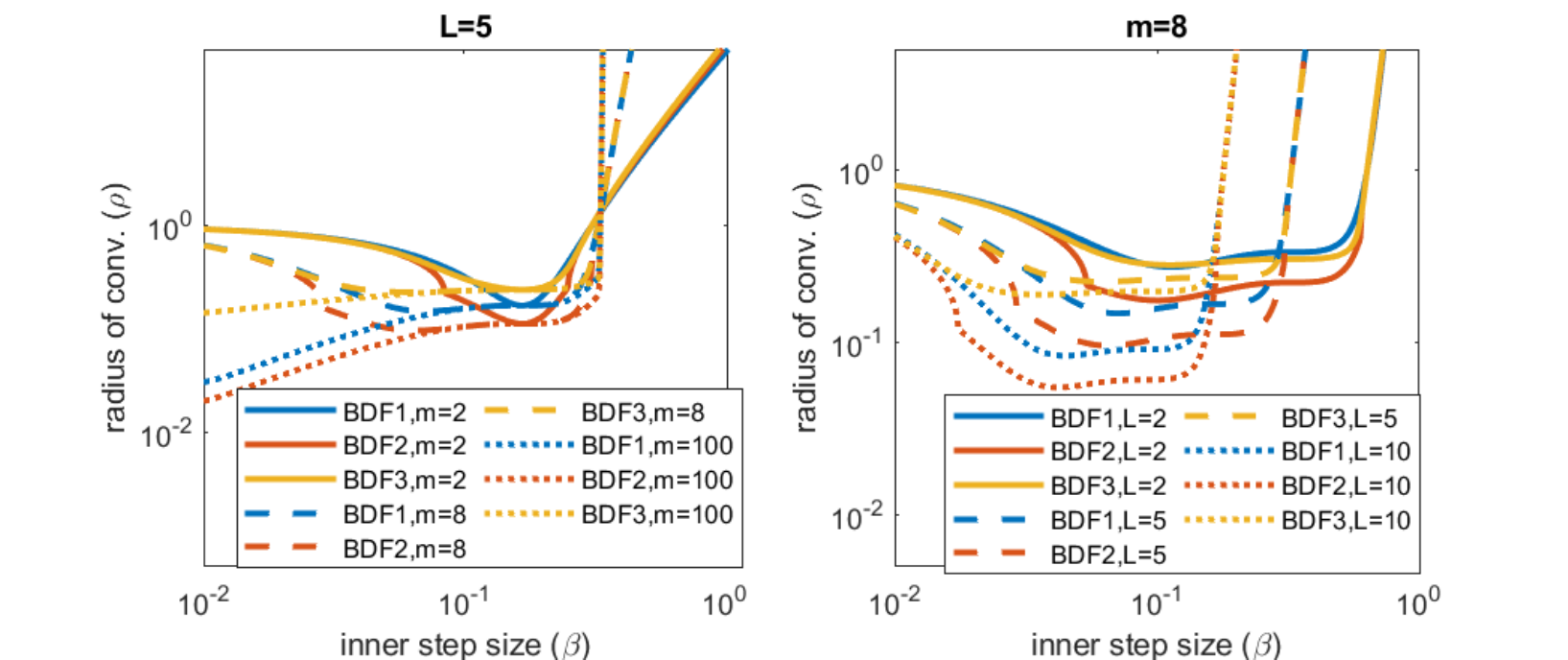}
    \caption{\textbf{Multistep radius of convergence.} $\lambda_{\max}(M)$ over changing values of $\beta$, for $\alpha = 1$ and varying $L = \lambda_{\max}(Q)$ and $m$ = number of inner gradient steps.}
    \label{fig:multistep_poly_stability}
\end{figure*}


\section{Convergence over general functions}
\label{sec:conv_general}
We now consider the convergence behavior of the multistep method over general functions $F$ which are $\mu$-convex (where $\mu < 0$ is possible) and (when needed) $L$-smooth.  
We show that for \emph{general $\xi_i$} and a wide range of functions, converges at similar convergence trends as single-step (approximate) proximal methods. 
A summary of the results in this section is given in Table \ref{tab:convergence_results}.

\begin{table}[htbp]
\renewcommand{\arraystretch}{1.5} 
\centering
\begin{tabular}
{|l|l|c|c|}
\hline
&& \textbf{Exact} & \textbf{Inexact} \\
\hline
\multirow{3}{*}
{\parbox{1.6cm}{s. convex \\ ($\mu\geq 0$)}}
& cnvg.&  $\|\vx-\vx^*\|=O((\frac{1 }{1+\beta \mu})^{k}) $ & $  \|\vx-\vx^*\|=  O((\gamma +\frac{(1+\gamma) }{1+\beta \mu} )^k )$ \\
&$\beta$ req. & none  & $\beta \geq \beta_{\min}$  \\
&Thm.  & \ref{th:stronglyconvex} & \ref{th:approx-strongly-convex}  \\
\hline
\multirow{3}{*}
{\parbox{1.6cm}{w. convex \\ ($\mu\leq 0$)}} &cnvg. & $\epsilon_\beta(\vx) = O(\sqrt{(1+\delta)/k})$
& $\|\nabla F(\vx)\|_2 = O(\sqrt{(1+\delta)/k})$
\\
&$\beta$ req. & $\beta < \frac{1-(\tau-1)\delta }{L}$ & $\beta < \frac{1-(\tau-1)\delta }{L}$ \\
&Thm.  & \ref{th:Non_convex_conv} & \ref{th:Nonconvex_approx_conv} \\
\hline
\end{tabular}
\caption{\textbf{Summary of convergence rates.}   Here, $\beta$ is the step size of the proximal point operator, $\tau$ is the multistep order, $\xi_i$ are the multistep coefficients, $\mu$ is the  convexity parameter, $L$ is the  smoothness parameter, $\gamma$ indicates prox inexactness,  $\delta$ is added error from nonconvexity, and $k$ is the iteration counter.
$\epsilon_\beta(\vx) = \Vert \frac{\prox_{\beta F}(\mathbf{x}) - \mathbf{x}}{\beta} \Vert$.
If $\beta$ req. is none, this means any $\beta \geq 0$ step size achieves the given convergence rate. Convergence rates are all given in Big-O. s. convex = strongly convex,  w. convex = weakly convex. cnvg = convergence rate. }
\label{tab:convergence_results}
\end{table}


We begin with the following Lemma which allows us to extend many of the proof techniques from the analyses of 
single-step proximal methods, to multi-step proximal methods.
\begin{lemma}\label{lem:sublinearseq}
Let $\{ a_k \}_{k=0}^{\infty}$ be a $\tau$-step sublinear sequence:
\begin{equation}
    a_k \leq \sum_{i=1}^{\tau} \xi_i a_{k-i}, 
    \label{eq:assp:sublinear}
\end{equation}
where   $k \geq \tau$. Then there exists $k >i_1 > i_2 > \cdots  i_{[\frac{k}{\tau}]}$ such that 
\begin{equation}
    a_k \leq  \left ( \sum_{i=1}^{k}|\xi_i| \right)  a_{i_1} \leq \left ( \sum_{i=1}^{k}|\xi_i| \right)^2  a_{i_2} \leq 
     \cdots\leq  \left ( \sum_{i=1}^{k}|\xi_i |\right)^{\lceil\frac{k}{\tau}\rceil}  a_{i_{\lceil \frac{k}{\tau}\rceil }}
     \label{eq:lem:sublinear}
\end{equation}
for all $k \in \{1,2,\cdots,\lceil\frac{k}{\tau}\rceil \}$.
\end{lemma}

\begin{proof}[Proof of Lemma \ref{lem:sublinearseq}]
We prove this result by mathematical induction. When $k = \tau$, by using Holder's inequality for $p=1,q=\infty$, we will have:
\begin{align*}
    a_{\tau} \leq  \sum_{i=1}^{\tau} \xi_i a_{k-i} \leq \left (\sum_{i=1}^{\tau}|\xi_i| \right) \left (\max_{i\in \{1,2,\cdots,\tau \}} a_{\tau-i} \right ).
\end{align*} In this first step, we pick $i_1 = \argmax{i\in \{1,2,\cdots,\tau \}} a_{\tau-i}$. This gives
\[
a_k \leq  \left ( \sum_{i=1}^{\tau}|\xi_i| \right)  a_{i_1}, \qquad i_1 = \argmax{i\in \{1,2,\cdots,\tau \}} a_{\tau-i}.
\]
Now applying \eqref{eq:assp:sublinear} 
 recursively, we get
\[
a_{i_{n}} \leq  \left ( \sum_{i=1}^{\tau}|\xi_i| \right)  a_{i_{n+1}},\qquad i_{n+1} = i_k-j, \qquad j =  \argmax{i\in 1,2,...,\tau}\; a_{i_n-i}.
\]
This inequality can be chained for as long as there exists $i_n \geq 0$. Note that at each step, the smallest value $i_n$ can take is $k - \tau n$. So, the largest value of $k$ such that $i_n \geq 0$ is guaranteed is $n = \lceil\frac{k}{\tau}\rceil$.
\end{proof}

\begin{remark}
Lemma \ref{lem:sublinearseq} holds for some initial $i_{[\frac{n}{\tau}]} \in \{0,1,2,\cdots,\tau-1\}$.
\end{remark}

\subsection{Exact proximal point steps}
We first consider the case of exact proximal point iterations. That is, we minimize $F(\vx)$ via the sequence

\begin{equation}
  \vx^{(k + 1)} =  \prox_{\beta F }(\tilde{\vx}^{(k)}), \qquad 
    \tilde{\vx}^{(k)} =  \sum_{i=1}^{\tau} \xi_i \vx^{(k-\tau+i)}.
     \label{eq:bdf-proxpt}
\end{equation}
\begin{theorem} \label{th:mu_strong_conv}
Let $F(\vx)$ be a $\mu$-strongly convex function. We minimize $F(\vx)$ via  \eqref{eq:bdf-proxpt} where  $\sum_{i=1}^\tau \xi_i = 1$ and   $\xi_i \geq 0$.
Then $\vx^{(k)}$ converges to $\vx^*$ linearly:
\[
\|\vx^{(k)}-\vx^*\|\leq \left (\frac{1}{1+\beta \mu} \right )^{\lfloor\frac{k}{\tau}\rfloor} \;\cdot\;  \max_{j \in \{ 0,1,\cdots \tau-1 \}} \Vert \vx^{(j)} - \vx^* \Vert.
\]
\label{th:stronglyconvex}
\end{theorem}
\begin{proof}
First, note that if $F$ is $\mu$-strongly convex, then  for $\zeta\in \partial F(\vx^{(k)})$,
  \begin{equation}
  \mu \|\vx^{(k)} - \vx^*\| \leq \|\zeta\| 
  \label{eq:mustrongconvex-1}
  \end{equation}
  and 
  \begin{equation}
0 \geq F(\vx^*)  -  F(\vx^{(k)})  \geq \langle \zeta, \vx^* -  \vx^{(k)} \rangle +  \frac{\mu}{2} \Vert \vx^{(k)} - \vx^* \Vert^2.
\label{eq:mustrongconvex-2}
\end{equation}
Now, the BDF step  \eqref{eq:bdf-proxpt} can be written as
\[
  \vx^{(k + 1)} =  \prox_{\beta F }(\tilde{\vx}^{(k)})=
     \sum_{i=1}^{\tau} \xi_i \vx^{(k+i-\tau)} - \beta \zeta_{k + 1}
\]
where $\zeta_{k + 1}\in \partial F(\vx^{(k+1)}) $. Rearranging, we see that  
\begin{eqnarray*}
 \Vert \sum_{i=1}^{\tau}\xi_i (\vx^{(k+i-\tau)}-\vx^*)\Vert ^2 
& =&\Vert \vx^{(k+1)} - \vx^* \Vert^2 + \beta^2 \Vert  \zeta_{k + 1} \Vert^2 
  + 2\beta \langle \zeta_{k + 1}, \vx^{(k+1)} - \vx^*\rangle \\
  &\overset{\eqref{eq:mustrongconvex-1}}{\geq} &\Vert \vx^{(k+1)} - \vx^* \Vert^2 +\beta^2 \mu^2 \Vert \vx^{(k+1)} - \vx^*\Vert^2  \\
  &&\qquad +2\beta \langle \zeta_{k + 1}, \vx^{(k+1)} - \vx^*\rangle \\
& \overset{\eqref{eq:mustrongconvex-2}}{\geq}&
  (1+2\beta \mu+\beta^2 \mu^2) \Vert \vx^{(k+1)} - \vx^* \Vert^2.
\end{eqnarray*}

As a result, by convexity of the $\ell_2$ norm,
\begin{eqnarray}
\Vert \vx^{(k+1)} - \vx^* \Vert 
&\leq& \frac{1}{1+\beta \mu}  \sum_{i=1}^{\tau}\xi_i \Vert \vx^{(k+i-\tau)}-\vx^*\Vert.
\label{eq:strongconvex_1stepcontract}
\end{eqnarray}
Now by Lemma \ref{lem:nonexpensive}, we will have $i_1,i_2,\cdots,i_{\lceil\frac{k+1}{\tau}\rceil}$ such that 
\begin{align*}
    \Vert \vx^{(k+1)} - \vx^* \Vert &\leq  \left (\frac{1}{1+\beta \mu} \right )  \Vert \vx^{(i_1)} - \vx^* \Vert
     \leq \left (\frac{1}{1+\beta \mu} \right )^{\lceil\frac{k+1}{\tau}\rceil}\hspace{-1ex}  \max_{j \in \{ 0,1,\cdots \tau-1 \}} \Vert \vx^{(j)} - \vx^* \Vert \label{eq11}
\end{align*}
for all $k \in \{1,2,\cdots,\lceil\frac{n}{\tau}\rceil \}$.
\end{proof}
 
Although Thm. \ref{th:stronglyconvex} only shows the case that $\xi_i\geq 0$
we can still have
linear convergence 
when $\xi_i<0$
by setting $\beta$ large enough.

\begin{theorem}

Let $F(\vx)$ be a $\mu$-strongly convex function. We minimize $F(\vx)$ via a $\tau$-multistep proximal point  method \eqref{eq:bdf-proxpt} where  $\sum_{i=1}^\tau \xi_i = 1$.
 Then it will converge linearly
\[
\|\vx^{(k+1)}-\vx^*\|\leq \left (\frac{\sum_{i=1}^{\tau} \vert \xi_i \vert }{1+\beta \mu} \right )^{[\frac{k+1}{\tau}]}  \max_{j \in \{ 0,1,\cdots \tau-1 \}} \Vert \vx^{(j)} - \vx^* \Vert
\]
for 
$\beta \geq  \frac{\sum_{i=1}^{\tau} \vert \xi_i \vert-1}{\mu}$.

\label{th:stronglyconvex-anyxi}
\end{theorem}
\begin{proof}
Similar to \eqref{eq:strongconvex_1stepcontract}, we will have:
\begin{multline*}
\Vert \vx^{(k+1)} - \vx^* \Vert \leq \frac{1}{1+\beta \mu} \Vert \sum_{i=1}^{\tau}\xi_i (\vx^{(k+i-\tau)}-\vx^*)\Vert   \\
\leq  \frac{\sum_{i=1}^{\tau} \vert \xi_i \vert}{1+\beta \mu} \Vert  (\vx^{(i_1)}-\vx^*)\Vert   
\leq 
\left ( \frac{\sum_{i=1}^{\tau} \vert \xi_i \vert}{1+\beta \mu}   \right )^{\lceil\frac{k+1}{\tau}\rceil}  \max_{j \in \{ 0,1,\cdots \tau-1 \}} \Vert \vx^{(j)} - \vx^* \Vert
\end{multline*}
 where the index $i_1$ is as defined in Lemma \ref{lem:sublinearseq}.
 
\end{proof}

 When $\tau = 1$, Theorem \ref{th:mu_strong_conv} reduces to the well-known convergence rate of the proximal point method for $\mu$-strongly convex functions. 
For larger $\tau$, the overhead is a factor of $\tau$ in the number of iterates needed; however, for good choices of $\xi_i$ this is never observed.

The proof of convergence for strongly convex functions is consistent with our expectation, and matches our quadratic analysis as well. However, in general, we cannot expect $F$ to be strongly convex, or even at times convex. Thus, we now consider \emph{nonconvex $F$} through the condition of weak convexity.
First, we confirm that even if $F$ is not convex, the mapping $\prox_{\beta F}$ is still contractive, as long as $\beta$ is suitably small.  This is intuitive, since as $\beta \to 0$, the proximal regularization term will overwhelm the nonconvexity in $F$. 
\begin{lemma} [\cite{hoheiselproximal}]\label{lem:nonexpensive}
Suppose that $F(\vx)$ is $\mu$-weakly convex function, and suppose $\beta$ is such that $0 < \beta \mu < 1$. Then
\[
\Vert \prox_{\beta F}(\vx) -  \prox_{\beta F}(\vy) \Vert < \frac{1}{1-\beta \mu} \Vert \vx - \vy \Vert.
\]
\end{lemma}

Now we give the convergence result.

\begin{theorem}\label{th:Non_convex_conv}
Let $F(\vx)$ be a $\mu$-weakly convex function; e.g. $\mu \leq 0$.
We minimize $F(\vx)$ via  \eqref{eq:bdf-proxpt} where
$\sum_{i=1}^\tau \xi_i = 1$ and 
\[
 \delta = (\tau-1) \sum_{j=1}^{\tau-1} \left  (\sum_{i=1}^{j}(\tau-i )\xi_i^2 \right) < 1.
\]
Then, given the step size requirement 
$
0\leq \beta \leq 
\frac{
1-\delta}{-\mu}
$
 we will have:
\[
 \min_{0\leq s \leq k }   \epsilon_{\beta}(\vx^{(k)}) \leq \frac{1}{1-\mu\beta} \sqrt{\frac{2(F(\vx^{(\tau)}) - F(\vx^*))}{k\beta} + \frac{\delta\sum_{s=0}^{\tau-1} \Vert \vx^{(s+1)}-\vx^{(s)} \Vert^2  }{k\beta^2}}
\]
where $\vx^{(s)}$ is an $\epsilon$-stationary and
$\epsilon_\beta(\vx) = \|\frac{\prox_{\beta F}(\vx)-\vx}{\beta}\|$.
\end{theorem}

\begin{proof}
By $\mu$-convexity, we will have:

\begin{eqnarray*}
    F(\vx^{(k)}) &\geq& F(\vx^{(k+1)}) + \langle \zeta_{k+1},\,\vx^{(k)}-\vx^{(k+1)}\rangle + \frac{\mu}{2}      
     \Vert \vx^{(k+1)}-\vx^{(k)} \Vert^2 \\
     &\overset{\eqref{eq:bdf-proxpt}}{=} &F(\vx^{(k+1)}) + \frac{1}{\beta}\langle \tilde{\vx}^{(k)} -   \vx^{(k+1)},\,\vx^{(k)}-\vx^{(k+1)}\rangle + \frac{\mu}{2}      
     \Vert \vx^{(k+1)}-\vx^{(k)} \Vert^2 \\
     &\overset{(*)}{ =} &
    F(\vx^{(k+1)}) + \frac{1}{2\beta}\underbrace{\Vert \vx^{(k+1)} - \sum_{i=1}^{\tau}\xi_i \vx^{(k+i-\tau)}\Vert^2}_{= \beta^2 \Vert \zeta_{k+1}  \Vert^2} + \left(\frac{1}{2\beta}+\frac{\mu}{2}\right)\Vert \vx^{(k+1)} - \vx^{(k)} \Vert^2 \\
    &&\qquad 
    - \frac{1}{2\beta}\Vert \sum_{i=1}^{\tau-1} \xi_i(\vx^{(k+i-\tau)} -\vx^{(k)}) \Vert^2   
\end{eqnarray*}
where $\zeta_{k} \in \partial F(\vx^{(k)})$, and $(*)$ is $2a^Tb =\|a\|^2+\|b\|^2-\|a-b\|^2$.
Telescoping, 
\begin{multline*}
F(\vx^{(\tau)}) - F(\vx^*) \geq \frac{\beta}{2}\sum_{j=0}^k \Vert \zeta_{j + 1} \Vert^2 + \\ 
\left(\frac{1}{2\beta}+\frac{\mu}{2}\right)\sum_{j=0}^k \Vert \vx^{(j+1)} - \vx^{(j)} \Vert^2 - \frac{1}{2\beta}\sum_{j=0}^k \Vert \sum_{i=1}^{\tau-1} \xi_i(\vx^{(j+i)} -\vx^{(j+\tau)}) \Vert^2.
\end{multline*}
For any $n\geq 1$, 
$\|\sum_{i=1}^n a_i\|^2 \leq n\sum_{i=1}^n \|a_i\|^2$.
Indicate this fact as $(*)$.
Then
\begin{align*}
    \Vert \sum_{i=1}^{\tau-1} \xi_i(\vx^{(k+i)} -\vx^{(k+\tau)}) \Vert^2 &\overset{(*)}{\leq} (\tau-1) \sum_{i=1}^{\tau-1}\xi_i^2\Vert \vx^{(k+i)}-\vx^{(k+\tau)} \Vert^2 \\
    &\overset{(**)}{\leq} (\tau-1) \sum_{i=1}^{\tau-1}\xi_i^2\|\sum_{j=i}^{\tau-1}\vx^{(k+j)} - \vx^{(k+j-1)}\|^2 \\
    &\overset{(*)}{\leq} 
     (\tau-1) \sum_{i=1}^{\tau-1}(\tau-i )\sum_{j=i}^{\tau-1}\xi_i^2\Vert \vx^{(k+j)}-\vx^{(k+j+1)} \Vert^2 \\
    &= (\tau-1) \sum_{j=1}^{\tau-1}  \left  (\sum_{i=1}^{j} (\tau-i)\xi_i^2 \right)\Vert \vx^{(k+j)}-\vx^{(k+j+1)} \Vert^2 \\
\end{align*}
where $(**)$ is from reverse telescoping
\[
\vx^{(k+i)} - \vx^{(k+\tau)} = \sum_{j=i}^{\tau-1}(\vx^{(k+j)} - \vx^{(k+j-1)}).
\]
Hence for $c_j =  \sum_{i=1}^{j} (\tau-i)\xi_i^2 $ we have:
\begin{multline*}
 \sum_{s=0}^{k} \Vert \sum_{i=1}^{\tau-1} \xi_i(\vx^{(s+i)} -\vx^{(s+\tau)}) \Vert^2 
 \leq (\tau-1) \sum_{r=0}^{k}\sum_{j=1}^{\tau-1} c_j \Vert \vx^{(r+j)}-\vx^{(r+j+1)} \Vert^2  \\ 
 \overset{s=r+j}{=} (\tau-1)\sum_{j=1}^{\tau-1} c_j  \sum_{s=j}^{j+k}\Vert \vx^{(s)}-\vx^{(s+1)} \Vert^2   
  \leq \underbrace{(\tau-1)    \sum_{j=1}^{\tau-1}  c_j }_{=\delta}\sum_{s=0}^{k-1} \Vert \vx^{(s+1)}-\vx^{(s)} \Vert^2\\ 
\end{multline*}
and thus
\[
\frac{1}{2\beta}\sum_{j=0}^k \Vert \sum_{i=1}^{\tau-1} \xi_i(\vx^{(j+i)} -\vx^{(j+\tau)}) \Vert^2  \leq  \frac{\delta}{2\beta}\sum_{s=0}^{k+\tau-1} \Vert \vx^{(s+1)}-\vx^{(s)} \Vert^2.
\]
As a result, we will have:
\begin{multline*}
F(\vx^{(\tau)}) - F(\vx^*) \geq \frac{\beta}{2}\sum_{i=\tau}^k \Vert\zeta_{i + 1} \Vert^2 \\ + \left(\frac{1}{2\beta}+\frac{\mu}{2}-\frac{
\delta}{2\beta}\right)\sum_{i=\tau}^k \Vert \vx^{(i+1)} - \vx^{(i)} \Vert^2 - \frac{
\delta}{2\beta}\sum_{s=0}^{\tau-1} \Vert \vx^{(s+1)}-\vx^{(s)} \Vert^2.
\end{multline*}
Now, we pick $\beta,\delta$ such that $\frac{1}{2\beta}+\frac{\mu}{2}- \frac{
\delta}{2\beta} \geq 0$ (e.g. $\beta < \frac{
\delta-1}{\mu}$). Then
\begin{align*}
\frac{1}{k}\sum_{i=0}^k \Vert \zeta_{i+\tau + 1}\Vert^2   &\leq \frac{2}{k\beta}(F(\vx^{(\tau)}) - F(\vx^*)) + \frac{\delta}{ k\beta^2}\sum_{s=0}^{\tau-1} \Vert \vx^{(s+1)}-\vx^{(s)} \Vert^2. 
\end{align*}
By the nonexpansiveness of the proximal operator from Lemma \ref{lem:nonexpensive}
\begin{align*}
\beta \epsilon_\beta(\vx^{(k+1)}) &=
\Vert \vx^{(k + 1)} -  \prox_{\beta F} ( \vx^{(k + 1)} ) \Vert \\
&= \Vert  \prox_{\beta F} ( \vx^{(k)}) -  \prox_{\beta F} ( \vx^{(k + 1)} ) \Vert  \\
& \leq \frac{1}{1-\mu\beta}\Vert \vx^{(k)} -  \vx^{(k + 1)}  \Vert = \frac{\beta}{1-\mu\beta}\Vert  \zeta_{k + 1} \Vert. 
\end{align*}
Since the minimum is smaller than the average, combining the previous two inequalities gives the result. 
\end{proof}

The biggest disadvantage in the weakly convex case is the restriction on the step size, which becomes more stringent with more negative  $\mu$ (indicating weaker convexity). Overall, $\delta$ also acts as an added noise term, although the overall rate is still $O(\sqrt{(1+\delta)/k})$, provided a step size is feasible ($1>\delta$). It is not surprising that \emph{some} restriction on the amount of nonconvexity should be present, as adding too large of a  regularization in the proximal step will warp the optimization landscape. 

\subsection{Inexact proximal steps}

Although Theorem \ref{th:stronglyconvex} give linear convergence rates when the step size is large enough, the exact multi-step proximal point methods are in general  hard to solve to completion, and is usually approximated in practice. Next, we give the linear convergence rate for aproximate multi-step proximal point methods, e.g. where $\vx^{(k+1)}$ satisfies \eqref{eq:approx_contract_assp} and 
\begin{equation}
  \vx_*^{(k + 1)} =  \prox_{\beta F }(\tilde{\vx}^{(k)}), \qquad 
    \tilde{\vx}^{(k)} =  \sum_{i=1}^{\tau} \xi_i \vx^{(k-\tau+i)}.
     \label{eq:approx-prox-bdf}
\end{equation}

\begin{theorem}
\label{th:approx-strongly-convex}
Let $F(\vx)$ be a $\mu$-strongly convex  function.  
We minimize $F(\vx)$ via   \eqref{eq:approx-prox-bdf} with $\sum_i\xi_i = 1$, and $\vx^{(k)}$ satisfies $\gamma$-contractiveness \eqref{eq:approx_contract_assp} for $\gamma < 1$.  Then $\vx^{(k)}$ converges to $\vx^*$ linearly 
\[    \Vert \vx^{(k+1)} - \vx^* \Vert   \leq \left ( \gamma +\frac{(1+\gamma)\eta}{1+\beta \mu} \right)^{\lceil\frac{k+1}{\tau}\rceil}  \,\cdot\,\max_{j \in \{ 0,1,\cdots \tau-1 \}} \Vert \vx^{(j)} - \vx^* \Vert
    \]
    where  $\eta =  \sum_{i=1}^\tau |\xi_i| $.
Convergence is guaranteed when $\gamma < \frac{\beta \mu - \eta + 1 }{\beta \mu + \eta + 1}$.

\end{theorem}
\begin{proof}
We compose a sequence of contractions:
\begin{eqnarray*}
    \Vert \vx^{(k +1)} - \vx^* \Vert & \leq&  \Vert \vx^{(k +1)} - \vx_*^{(k+1)} \Vert  + \Vert \vx_*^{(k+1)} -\vx^{*} \Vert \\
    &\overset{\eqref{eq:approx_contract_assp}}{\leq} & \gamma     \Vert  \tilde\vx^{(k)}  - \vx_*^{(k+1)}\Vert + \Vert \vx_*^{(k+1)} -\vx^{*} \Vert\\
    & \le& \gamma \left( \Vert   \tilde\vx^{(k)} - \vx^*\Vert +\Vert \vx_*^{(k+1)}  - \vx^*\Vert \right ) +\Vert \vx_*^{(k+1)} -\vx^{*} \Vert\\
    & \overset{\text{Th. \ref{th:mu_strong_conv}}}{\le}& \gamma \left( \Vert    \vx^{(i_1)} - \vx^*\Vert +\frac{\eta}{1+\beta \mu} \Vert \vx_*^{(i_1)}  - \vx^*\Vert \right ) +\frac{\eta}{1+\beta \mu} \Vert \vx_*^{(i_1)} -\vx^{*} \Vert\\
   &=& (\gamma +\frac{(\gamma + 1)\eta}{1+\beta \mu} )\Vert  \vx^{(i_1)}-\vx^*\Vert  \\
\end{eqnarray*}
and $i_1$ is as constructed in Lemma \ref{lem:sublinearseq}. Applying Lemma \ref{lem:sublinearseq} gives the desired result. 
\end{proof}

Theorem \ref{th:approx-strongly-convex} is unsurprising, and follows the same trend as that of the quadratic problems. Indeed, if $\gamma = 0$, we recover the exact rate (Th. \ref{th:stronglyconvex}, \ref{th:stronglyconvex-anyxi}).

We now consider the approximate multistep proximal point method applied to nonconvex functions. Here, a more stringent step size requirement is needed.
\begin{theorem}\label{th:Nonconvex_approx_conv}
Let $F(\vx)$ be a $\mu$-weakly convex function; i.e. $\mu \leq 0$.
We minimize $F(\vx)$ via \eqref{eq:approx-prox-bdf}, where $\vx^{(k)}$ satisfies $\gamma$-contractiveness \eqref{eq:approx_contract_assp} for $\gamma < 1$. Assume $\xi_i$ is such that    
$\sum_{i=1}^\tau \xi_i = 1$ and 
\[
\delta = (1 + \gamma^2 (1+\beta \mu)^2 )\delta_{*} < \frac{1}{2},\quad  
\delta_{*} = (\tau - 1)\sum_{j=1}^{\tau-1} \left  (\sum_{i=1}^{j}(\tau-i )\xi_i^2 \right).
\] 
Then, given the step size requirement 
$
0\leq \beta \leq 
\frac{2-4\delta}{-\mu}
$
 we will have:
\[
    \min_{0\leq s \leq k } \Vert \nabla F(\vx^{(k + 1)})\Vert \leq  \sqrt{\frac{2(F(\vx^{(\tau)}) - F(\vx^*))}{k\beta} + \frac{\delta}{k\beta^2 }\sum_{s=0}^{\tau-1} \Vert \vx^{(s+1)}-\vx^{(s)} \Vert^2}.
\] 
\end{theorem}
\begin{proof} 
Recall that the proximal operation's optimality condition ensures 
\begin{equation}
\vx_*^{(k+\tau + 1)} =  
    \tilde{\vx}^{(k+\tau)} - \beta \nabla F(\vx_*^{(k+\tau+1)}).
     \label{eq:prox-opt-cond}
\end{equation}
By the $\mu$-convex property, we have 
\begin{eqnarray*}
    F(\vx^{(k)}) &\geq& F(\vx^{(k + 1)}) +   \langle  \nabla F(\vx^{(k + 1)}),\vx^{(k)}-\vx^{(k + 1)}\rangle + \frac{\mu}{2}      
     \Vert \vx^{(k + 1)}-\vx^{(k)} \Vert^2\\
     &=&
    F(\vx^{(k + 1)}) + \frac{\beta}{2} \Vert\nabla F(\vx^{(k + 1)}) \Vert^2 + 
     \left(\frac{1}{2\beta}+\frac{\mu}{2}\right)\Vert \vx^{(k + 1)}  - \vx^{(k)} \Vert^2  \\ 
     &&\qquad  - \frac{1}{2\beta}\Vert \vx^{(k)} -\vx^{(k + 1)} - \beta \nabla F(\vx^{(k + 1)}) \Vert^2.  
\end{eqnarray*}
Take 
\begin{eqnarray*}
A_k &=& \vx^{(k)} -\vx^{(k + 1)} - \beta \nabla F(\vx^{(k + 1)}),\\
B_k &=&\vx_{*}^{(k+1)}  - \sum_{i}\xi_i \vx^{(k+i-\tau)} + \beta  \nabla F(\vx_*^{(k+1)}),\\
C_k &=& \sum_{i=1}^{\tau-1} \xi_i(\vx^{(k+i-\tau)} -\vx^{(k)}),\\
D_k &=& \vx^{(k + 1)} + \beta \nabla F(\vx^{(k + 1)}) - \vx_{*}^{(k+1)}  - \beta  \nabla F(\vx_*^{(k + 1)}).
\end{eqnarray*}
Then $A_k+B_k=C_k+D_k$.
By \eqref{eq:prox-opt-cond},  $B_k = 0$. So, 
\[
 \|A_k\|^2 =\|A_k+B_k\|^2  = \|C_k+D_k\|^2 \leq 2\|C_k\|^2+2\|D_k\|^2.
 \]
In addtion, by $\gamma$-contraction and $\mu-$ smoothness, we have 
\begin{eqnarray*}
    \|D_k\|
    &\leq & 
      \Vert \vx^{(k + 1)} - \vx_{*}^{(k+1)} \| + \beta \|  \zeta_{k + 1}  -  \zeta^{*}_{k}  \|
      \leq 
     \gamma (1+\beta L)\|C_k \|.
\end{eqnarray*}
And, similar to Theorem \ref{th:Non_convex_conv}, we have 
\begin{eqnarray*}
\sum_{s=0}^k \|C_{s+\tau}\|^2
&\leq& \underbrace{(\tau-1)    \sum_{j=1}^{\tau-1}  \sum_{i=1}^{j} (\tau-i)\xi_i^2 }_{=\delta_*}\sum_{s=0}^{k-1} \Vert \vx^{(s+1)}-\vx^{(s)} \Vert^2\\
\end{eqnarray*}
and thus
\[
\sum_{s=0}^k \|C_{s+\tau}\|^2 + \|D_{s+\tau}\|^2=
\underbrace{(1+\gamma^2(1+\beta L)^2)\delta^* }_{\delta }\sum_{s=0}^{k-1}\|\vx^{(s+1)}-\vx^{(s)}\|^2.
\]
Therefore, via telescoping from $k = \tau$ to termination, we will have:
\begin{multline*}   
F(\vx^{(\tau)}) - F(\vx^*) \geq \frac{\beta}{2} \sum_{i=0}^k \Vert \nabla F(\vx^{(i + 1)})\Vert^2 \\
 + \left(\frac{1-2\delta}{\beta}+\frac{\mu}{2}  \right )\sum_{i=0}^k \Vert \vx^{(i+1)} - \vx^{(i)} \Vert^2 
- \frac{\delta}{2\beta}\sum_{s=0}^{\tau-1} \Vert \vx^{(s+1)}-\vx^{(s)} \Vert^2.
\end{multline*}
As a result, provided $\beta \leq \frac{2-4\delta}{-\mu}$, we have
\begin{align*}
\frac{1}{k} \sum_{k=0}^k \Vert \nabla F(\vx^{(k + 1)})\Vert^2 &\leq \frac{2(F(\vx^{(\tau)}) - F(\vx^*))}{k\beta} + \frac{\delta}{k\beta^2 }\sum_{s=0}^{\tau-1} \Vert \vx^{(s+1)}-\vx^{(s)} \Vert^2.
\end{align*}
\end{proof}

Taking $\gamma = 0$ would indeed result in an exact proximal rate, but it is not as tight as that given in Theorem \ref{th:Non_convex_conv}; in particular, the restriction on $\delta$ is more stringent by a factor of $1/2$. However, the proof technique is indeed very similar, and the overall trend of $O(\sqrt{(1+\delta)/k})$ is still similar. 
Overall, the rates in this section is consistent with those for single-step approximate proximal methods, and indicates that \emph{under very general construction of $\xi_i$}, stability and convergence are not only guaranteed, but at worse only  constant factor $\tau$ less optimal than the single-step version. 

In the next section, we will see that in applications, the $\tau$-multistep methods are actually \emph{improvements} upon single-step methods, especially in applications where instability is observed. We will revisit theoretical rates again with an \emph{accelerated} rate for the specific application of alternating projections.

\section{Applications}
\label{sec:applications}

%
In this section, we empirically validate our proposed methods by considering several optimization problems: proximal gradient with $\ell_1$ norm, proximal gradient with LSP penalty,  alternating minimization for matrix factorization, and alternating projection on linear subspaces.

\subsection{Multi-step prox gradient} The  first approach approximates
\begin{align}
\label{eq:proximal_mapping}
 \prox_{\alpha f + \alpha h}(\tilde \vx^{(k)}) \approx \prox_{ \alpha h}(\tilde \vx^{(k)} - \alpha \nabla f(\tilde \vx^{(k)}))  .
\end{align}
Here, $\tilde \vx^{(k)}$ is defined using the BDF mixing coefficients as defined in \eqref{eq:bdf-mixing} and Table \ref{tab:bdf-const}.  
Note that the major multi-step extension is in the construction of $\tilde \vx^{(k)}$; specifically, the \emph{per-iteration complexity} of computing $\tilde \vx^{(k)}$ differently does \emph{not} significantly add computational overhead -- but it does affect convergence behavior.

\paragraph{Proximal gradient with 1-norm.}
(Figure \ref{fig:Prox-L1}.)
The proximal gradient with $\ell_1$ norm over compressed sensing problem is formulated as:
\begin{align*}
    \min_{\vx 
    \in \mathbb{R}^q}\; \frac{1}{2} \Vert \vA \vx - \vb \Vert^2 +\lambda \Vert \vx \Vert_1
\end{align*}
where $\vA \in \mathbb{R}^{p \times q}$ and $\vb \in \mathbb{R}^p$, with   $ p  \ll q$ (underdetermined system) and the $\ell_1$ norm is to encourage sparsity in the solution. We choose $p = 50$ and $ q = 100$ and various choices of $\lambda$. 
In this and the next LSP problem, we construct $\vA$ to have three types of conditioning, with $\vA = U\Sigma V^T$, for  randomly chosen orthogonal $U$ and $V$, and $\Sigma = \diag(\sigma_1,...,\sigma_p)$ where 
we distribute $\sigma_r$ as uniformly distributed (very good conditioning), $1/r$ (bad conditioning), or $\exp(-r)$ (very bad conditioning) for $r = 1,...,p$. 
%
%
%
In all cases, we see consistent improvement of the higher-order BDF scheme, across all conditioning, and with more improvement when $\lambda$ is larger.

\begin{figure}[!ht]
    \centering
 \includegraphics[width=.8\linewidth]{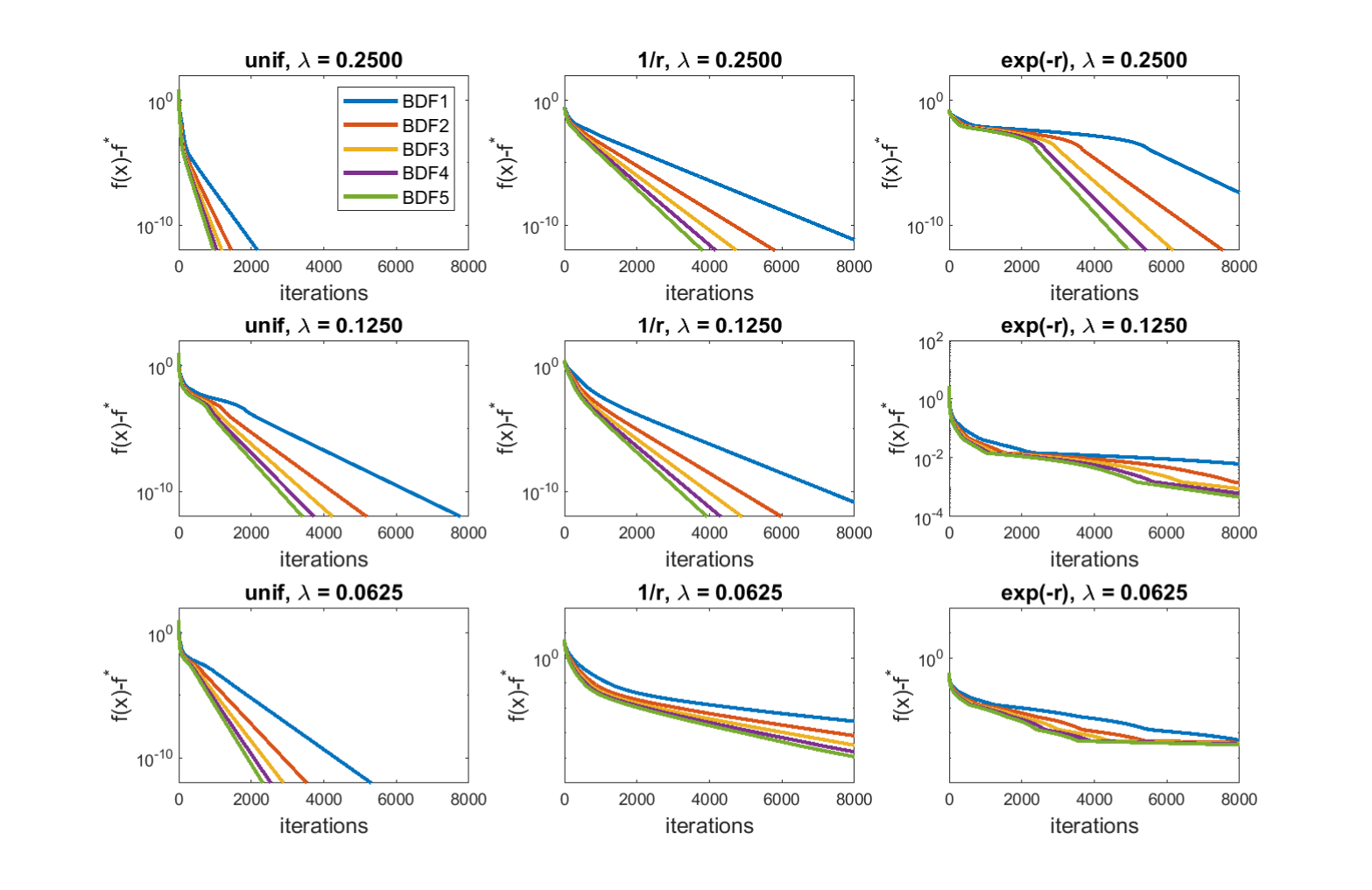}
    \caption{\textbf{Comparsion of different BDF schemes for   proximal gradient with $\ell_1$ penalty. }
    Data is constructed with singular values uniformly distributed (unif), with inverse decay ($1/r$), or with exponential decay ($\exp(-r)$) to simulate sensing under different sensing matrix rank conditions.
     }
    \label{fig:Prox-L1}
\end{figure}

\paragraph{LSP.}
(Figure \ref{fig:Prox-LSP}.)
The proximal gradient with LSP penalty over compressed sensing is formulated as follows:
\begin{align*}
    \min_{\vx 
    \in \mathbb{R}^q} \frac{1}{2} \Vert \vA \vx - \vb \Vert^2 +  \sum_{i=1}^q \log\left(1+\frac{|x_i|}{\theta}\right)
\end{align*}
where  $\vA \in \mathbb{R}^{p \times q}$ and $\vb \in \mathbb{R}^p$. 
Here, $p = 20$ and $q = 50$ (and the system is underdetermined).
We use  \eqref{eq:lspflow}
to compute the proximal operator of the LSP penalty. 
Compared to the $\ell_1$ norm, the nonconvex LSP norm is a more aggressive sparsifier.
Usually, $\theta$ can be used to control the amount of concavity near 0, making sparsification more aggressive. 
%
%
%
%
%
%
 Similar to $\ell_1$ norm setting, the higher-order BDF performs better than lower order BDF scheme. However, there is a more marked improvement at even smaller precisions; the improvement seems faster than linear.

\begin{figure}[!ht]
    \centering
    \includegraphics[width=.8\linewidth]{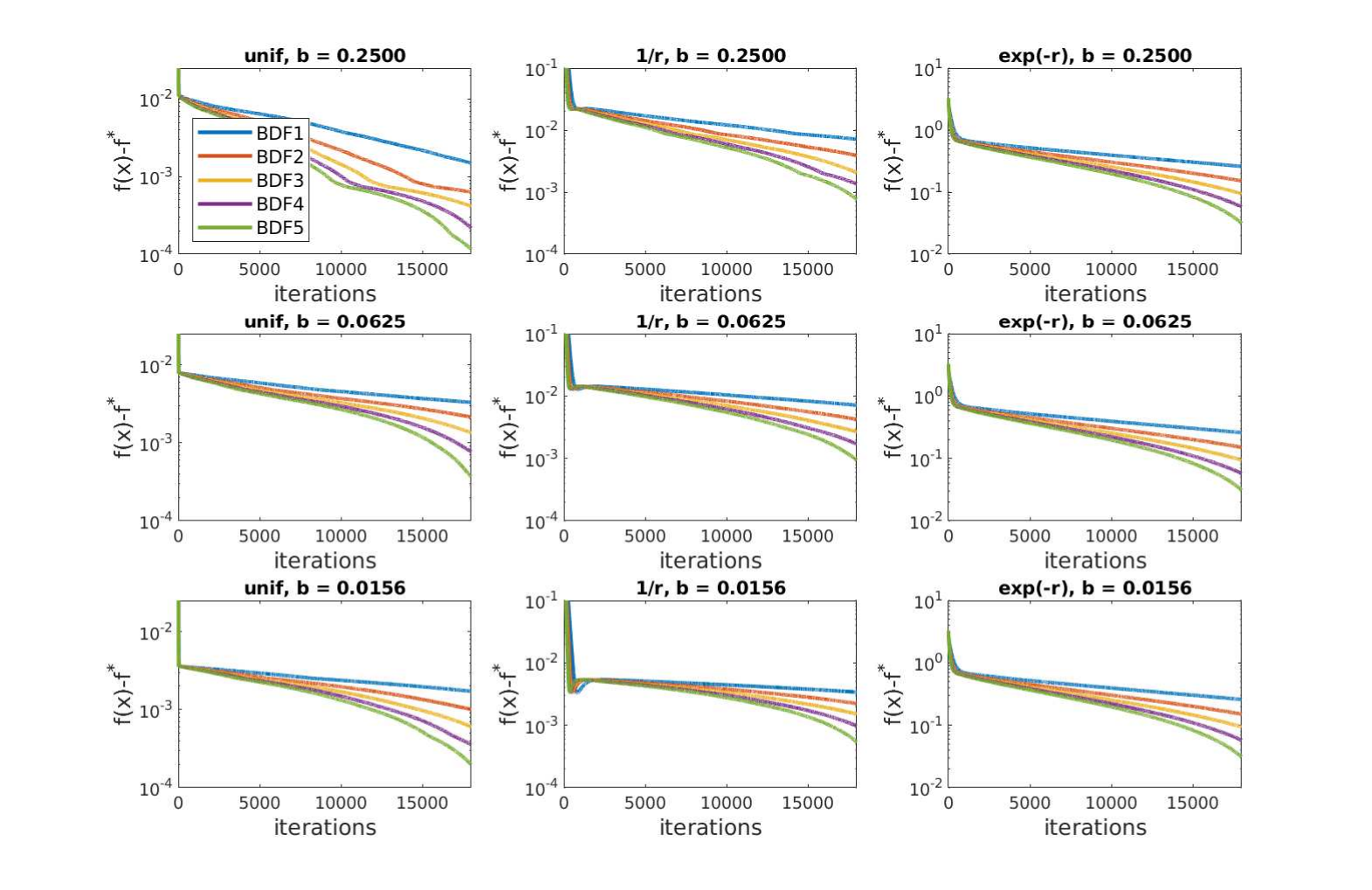}
    \caption{\textbf{Comparsion of different BDF schemes for proximal gradient with   LSP (nonconvex) penalty.}  Data is constructed with singular values uniformly distributed (unif), with inverse decay ($1/r$), or with exponential decay ($\exp(-r)$) to simulate sensing under different sensing matrix rank conditions.
     }
    \label{fig:Prox-LSP}
\end{figure}

\subsection{Multi-step alternating projections}
The next approach is to minimize the projection on an intersection of subspaces
\begin{equation}
 \min_{\vx\in \mathcal{C}_1  \bigcap \mathcal{C}_2} \quad \|\vx-\vx_0\|^2
 \label{eq:projsubspaces}
\end{equation}
where $\mathcal C_1$ and $\mathcal C_2$ are closed linear subspaces. 
The approach is inspired by the alternating projection method \citep{von1949rings}  in which for $\vx^{(0)} = \vx_0$,  performs 
\begin{equation}
\vy^{(k+1)} = \proj_{\mathcal{C}_1}({\vx}^{(k)}), \qquad 
\vx^{(k+1)} = \proj_{\mathcal{C}_2}(\vy^{(k + 1)}),
\label{eq:altproj}
\end{equation}
and as $k\to +\infty$, $\vx^{(k)},\vy^{(k)}$ converge to the solution of \eqref{eq:projsubspaces}.
In other words, for the breakdown
\[
f(\vx) = \frac{1}{2}\|\vx-\vx_0\|^2, \qquad h(\vx) = \mI_{\mC_1\cap \mC_2}
\]
we view \eqref{eq:altproj} as approximating
\[
\prox_{\alpha f + \alpha h}(\tilde \vx) = \proj_{\mC_1\cap \mC_2}(\frac{\alpha \vx_0 +  \tilde \vx}{\alpha+1})\overset{\alpha\to 0}{\to} \proj_{\mC_1\cap \mC_2}(\tilde  \vx)
\]
and we run 
\begin{align}
\vy^{(k+1)} = \proj_{\mathcal{C}_1}(\tilde {\vx}^{(k)}),\qquad 
\vx^{(k+1)} = \proj_{\mathcal{C}_2}(\vy^{(k + 1)}).   
\label{eq:altproj3}
\end{align}

Our results on alternating projects is in Figure \ref{fig:Prox-altproj}.
Here, we take 
$\mathcal{C}_1$ and $\mathcal{C}_2$ to be the column space of matrix $\mathbf{C}_1 $, $\mathbf{C}_2\in \mathbb R^{500\times 400}$. We first produce matrices $\mathbf{C}_1$ and $\mb Z$ where each element is $\sim\mathcal N(0,1)$, and set $\mb C_2 = (1-\sigma)\mb C_1 + \sigma \mb Z$. The size of $\sigma\in [0,1]$ controls the coherency between the two subspaces; if the coherency is high ($\sigma$ is small), then the problem becomes more ill-conditioned. In our experiments, we found that in general, higher $\tau$ produces faster convergence, and instability for high $\tau$ only seems to occur when $\sigma$ is large (conditioning is not bad).

\begin{figure}[!ht]
    \centering
        \includegraphics[width=.8\linewidth]{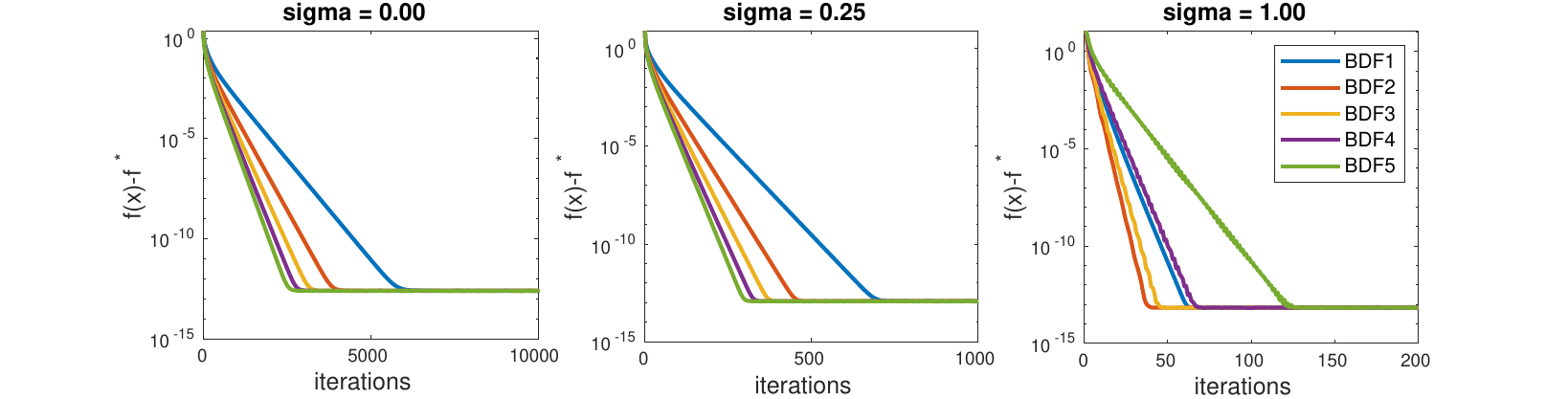}

    \caption{\textbf{Comparsion of different BDF schemes for alternating linear projections.} $\sigma$ is the noise parameter that controls the angle between the subspaces; smaller means more ill-conditioned. 
     }
   \label{fig:Prox-altproj}
\end{figure}

\subsection{Multi-step alternating minimization} The final approach is based on alternating minimization  \citep{shi2016primer,nesterov2012efficiency,liu2014asynchronous,fercoq2015accelerated} to solve a (possibly nonconvex) problem 
\[
F(\vx) = f(\vx_1,\vx_2), \qquad h(\vx) = 0.
\] 
The idea is to  interpret the (regularized) alternating minimization steps 
\begin{eqnarray*}
\vx_{1}^{(k+1)}& =& \argmin{\vx_1} \;f(\vx_1,\vx_2^{(k)}) + \frac{1}{2\alpha} \|\vx_1-\vx_1^{(k)}\|^2 \\
\vx_{2}^{(k+1)}& =& \argmin{\vx_2}\; f(\vx_1^{(k+1)},\vx_2 ) + \frac{1}{2\alpha} \|\vx_2-\vx_2^{(k+1)}\|^2\\
\end{eqnarray*}
as approximations of the  proximal point operations
\begin{multline*}
(\vx_{1}^{(k+1)} ,
\vx_{2}^{(k+1)}) \approx  \argmin{\vx_1,\vx_2} \; f(\vx_1 ,\vx_2 ) +\frac{1}{2\alpha} \|\vx_1-\vx_1^{(k)}\|^2 + \frac{1}{2\alpha} \|\vx_2-\vx_2^{(k)}\|^2.
\end{multline*}
That is, the approximation is done by splitting the variable into  blocks and minimizing each block separately.
The multi-step extension is then done by 
\begin{align*}
{\mathbf{x}}^{(k+1)}_1 &= \argmin{\mathbf{x}_1}\, f(\vx_1, \tilde{\mathbf{x}}^{(k)}_2)                         + \frac{1}{2\alpha} \Vert \mathbf{x}_1 - \tilde{\mathbf{x}}^{(k)}_1 \Vert^2,\\
{\mathbf{x}}^{(k+1)}_2 &= \argmin{\mathbf{x}_2} \,f({\mathbf{x}}^{(k+1)}_1, \vx_2 )                        + \frac{1}{2\alpha} \Vert \mathbf{x}_2 - \tilde{\mathbf{x}}^{(k)}_2 \Vert^2.
\end{align*}

\paragraph{Matrix factorization.}
 (Figure \ref{fig:Prox-altmin}.) 
 We apply the higher order alternating minimization scheme to the  matrix factorization problem:
\begin{align*}
    \min_{\vU,\vV}\; \frac{1}{2}\Vert \vU \vV^{T} - \vR \Vert_F^2.
\end{align*}
%
For the experiment settings, we choose $\vR$ as a $100 \times 100$ matrix and assign $\vU,\vV\in \R^{100\times r}$ random values in $\mathcal{N}(0,1)$, and where $r$ is the rank of the completed matrix. 
%
Here, we see some clear advantages for multstep order $\tau = 2,3$, but when $\tau = 4,5$ then there is considerably instability, and the performance degrades. This suggests that the restricted stability region of the higher order scheme is being violated in this problem setting. 
The overall convergence seems faster when $r$ is larger (better conditioned problems) and when $\alpha$ is larger (larger step sizes).

\begin{figure}[!ht]
    \centering
             
        \includegraphics[width=.8\linewidth]{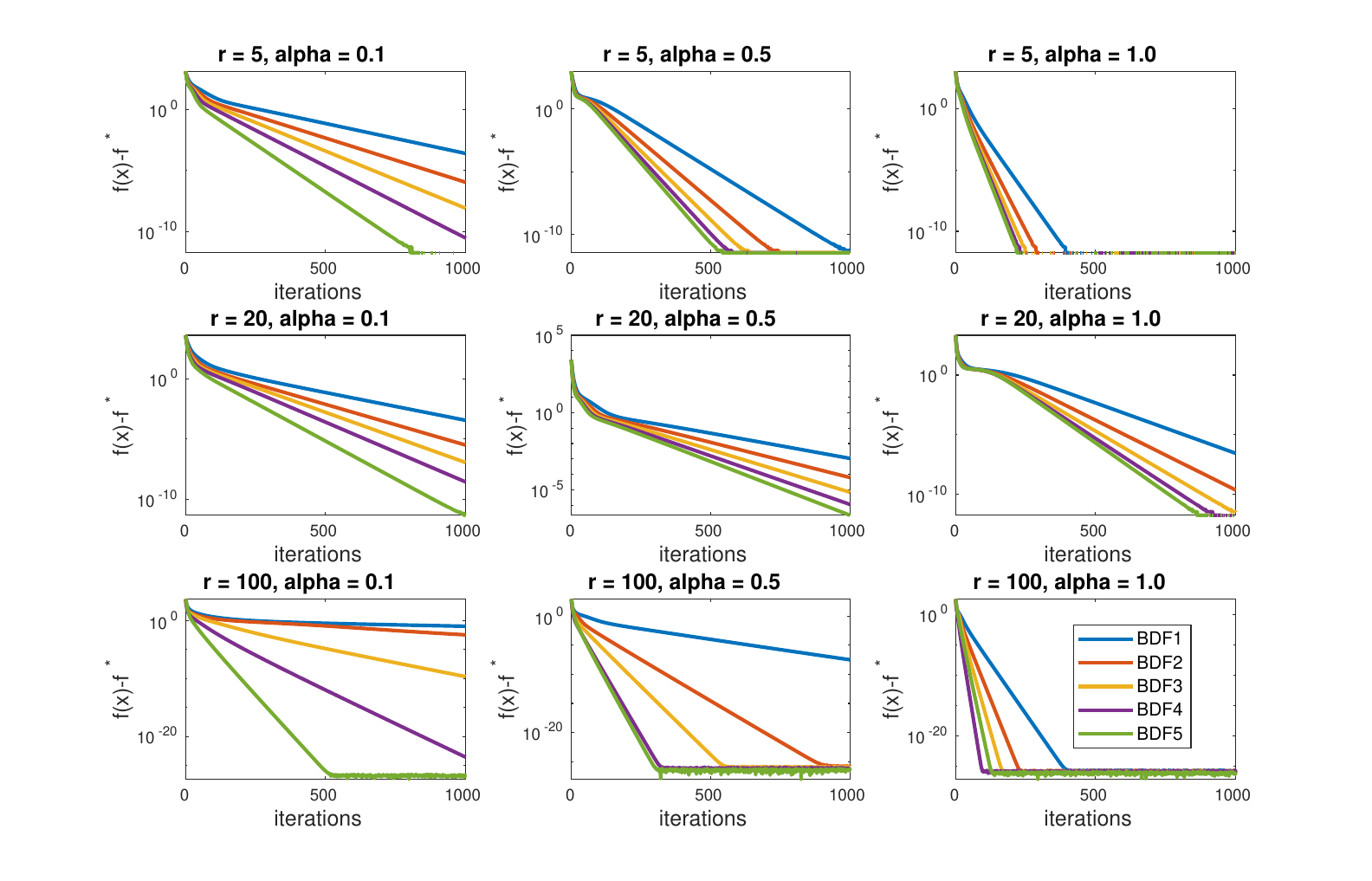}

    \caption{\textbf{Comparsion of  BDF schemes for    alternating minimizations for matrix factorization.} Here $r$ is the matrix rank  and $\alpha$ the step size of the inner loop.
     }
    \label{fig:Prox-altmin}
     
\end{figure}

\section{Acceleration of alternating projections}
\label{sec:altproj}
Finally, we demonstrate that the weights of BDF can be tuned for acceleration, in the case of the alternating projection method. Heere, we consider the projection on $\mC_1\cap \mC_2$ where $\Pi_1$ and $\Pi_2$ are the  idempotent projection matrices such that 
\[
\Pi_1 \mb x = \proj_{\mC_1}(\mb x), \qquad \Pi_2 \mb x = \proj_{\mC_2}(\mb x).
\]
Then, a vanilla alternating projections method performs 
\[
\vx^{(k+1)} = \Pi_1\Pi_2 \vx^{(k)} = (\Pi_1\Pi_2)^{k+1}\vx^{(0)}.
\]
We perform the Jordan canonical decomposition  for $M := \Pi_1\Pi_2 = VDV^{-1}$ where $D = J+\Lambda$, 
and $J$ is 0 everywhere except  possibly at its first offdiagonal ($J_{i,i+1}\in \{0,1\}$), and $\Lambda$ is a diagonal matrix containing the   eigenvalues of $M$.  
Since $\Pi_1$ and $\Pi_2$ are projection matrices, their eigenvalues are all $\in \{-1,1\}$, so the  eigenvalues of $M$ must also exist in this bound. 
Then
\[
\vr^{(k)} = (I-\Pi_1\Pi_2) \vx^{(k)} = (I-\Pi_1\Pi_2) (\Pi_1\Pi_2)^k \vx^{(0)} = V (I-D) D^k V^{-1}\vx^{(0)}
\]
and
\[
D^k = (\Lambda+J)^k = \Lambda^k + O(\Lambda^{k-1}).
\]
Now, define $\rho$ as the smallest nonzero generalized eigenvalue of $I - \Pi_1\Pi_2$, e.g. 
\[
\rho := \min_{i:\Lambda_{ii}\neq 0} 1-\Lambda_{ii}.
\]
 Then, it is clear that 
\begin{eqnarray*}
\|\vr^{(k)}\| &\leq& \rho ((1-\rho)^k + O((1-\rho)^{k-1}) )\|\vx^{(0)}\|_2\\
&\leq& \rho \|\vx^{(0)}\|_2\exp(-\rho k)+ O(\exp(-\rho(k-1))).
\end{eqnarray*}
That is, the constant $1/\rho$ is a measure of \emph{problem conditioning}; when $\rho$ is very small, it means that $\mC_1$ and $\mC_2$ have a subspace that is very close together. Note that in this regime,  the alternating projections operations tend to zig-zag. 
\begin{lemma}
    Alternating projections using $\tau$ multi-step updates converges at  rate $O(\exp(-\eta k))$ where $\eta = \max_{i,j}\eta_{i,j}$ is the largest root of 
    \[
\left(\sum_{s=2}^\tau \xi_s \eta_{i,j}^{2-s} +  \eta_{i,j}^{-(\tau-1)} \xi_1\right) \lambda_i = \eta_{i,j}, \quad j = 1,...,\tau,\quad i=1,...,n.
\]

\end{lemma}

\begin{proof}
Define $\tilde \xi = (\xi_2,...,\xi_\tau) \in \R^{\tau-1}$, and recall that the variable $\mb x^{(k)}\in \R^n$. 
     For multicase updates, define  
     \begin{eqnarray*}
     \vz^{(k)} &=&  \begin{bmatrix} (\mb x^{(k-\tau+1)})^T & (\vu^{(k)})^T \end{bmatrix}^T \in \R^{\tau n}, \\ 
     \vu^{(k)}&:=& \begin{bmatrix} (\vx^{(k-\tau+2)})^T & \cdots & (\vx^{(k)})^T\end{bmatrix}^T \in \R^{(\tau-1) n}.
     \end{eqnarray*}
     Here,  we initialize $\vz^{(0)} = (\vx^{(0)},..,\vx^{(0)})$.
     Then  $\mb x^{(k+1)} = \Pi_1\Pi_2 \tilde {\mb x}^{(k)} $ and
     \begin{eqnarray*}
      \tilde{\mb x}^{(k)} = \sum_{i=1}^\tau \xi_i \mb x^{(k-\tau+i)} = \xi_1 \mb x^{(k+\tau+1)}+(\tilde \xi^T\otimes I) \mb u^{(k)}, 
     \end{eqnarray*}
    can be written succinctly as 
 \[
 \vz^{(k+1)} = 
 \begin{bmatrix}
     \vx^{(k+1)}\\
     \vu^{(k)}
 \end{bmatrix}
 =
 \begin{bmatrix}
    \Pi_1\Pi_2 (\tilde \xi^T \otimes I) & \xi_1 \Pi_1\Pi_2  \\ 
     I_{n(\tau-1)} & 0
 \end{bmatrix}
 \begin{bmatrix}
     \vu^{(k)}\\
     \vx^{(k-\tau+1)}
 \end{bmatrix} =:A \mb z^{(k)}.
 \]
 Now  by Gelfand's rule, since
 \[
 \mb r^{(k)} = (I-\Pi_1\Pi_2)\mb x^{(k)} = (I-\Pi_1\Pi_2)(I\otimes \mb e_1^T)\mb z^{(k)} = (I-\Pi_1\Pi_2)(I\otimes \mb e_1^T)A^k\mb z^{(0)}
 \]
 then 
 \[
 \|\mb r^{(k)}\| \leq \lambda_{\max}(A)^k \rho \tau ^2 \|\mb x^{(0)}\|.
 \]
 It is left to derive $\lambda_{\max}(A)$. By using Schur complements, it can be shown that if $\eta$ is an eigenvalue of $A$, then the characteristic polynomial
\begin{eqnarray*}
\det(A - \eta I)
  &=& (-\eta)^{(\tau-1)n} \det( \underbrace{(\xi_2  
  + \eta^{-1} 
 \xi_3   + \cdots +  \eta^{-(\tau-1)} \xi_1) }_{\bar \xi (\eta) }\Pi_1\Pi_2 - \eta I)
\end{eqnarray*}
and 
\[
\det(A-\eta I) = 0 \iff \eta = 0 \text{ or } \det\left(\bar \xi(\eta)\Pi_1\Pi_2 -  \eta I\right) = 0.
\]
Define $\eta_{i,j}$ as the $j$th root of the following polynomial, for the $i$th eigenvalue of $\Pi_1\Pi_2$. Then it must be that 
$
\bar \xi(\eta_{i,j}) \lambda_i = \eta_{i,j}.
$
\end{proof}
In general, it is not clear how to identify these roots.
We offer a tighter analysis for $\tau = 2$.
\begin{theorem}
    The 2nd order multi-step method with coefficients
    \[
\xi_1 = \frac{-\rho  + \sqrt{\rho}}{1-\rho}, \qquad \xi_2 =  \frac{1 -  \sqrt{\rho}}{1-\rho}
\]
 achieves an accelerated rate
    \[
    \|\vr^{(k)}\| \leq  2\rho \|\vx^{(0)}\|\exp(-\sqrt{\rho} k).
    \]
\end{theorem}
 
\begin{proof}
    In the case of $\tau = 2$,  the characteristic polynomial becomes a quadratic equation
    \[
    \xi_2\lambda_k \eta +  \xi_1 \lambda_k = \eta^2
    \]
    may be written  explicitly as
\[
\eta_{k,1}, \eta_{k,2} = \frac{  \lambda_k\xi_2   \pm \sqrt{(\lambda_k\xi_2)^2 + 4 \lambda_k \xi_1 }}{2}.
\]
Specifically, picking
\[
\xi_1 = \frac{\lambda_k - 2 \mp 2\sqrt{1-\lambda_k}}{\lambda_k}, \qquad \xi_2 =  \frac{2 \pm  2\sqrt{1-\lambda_k}}{\lambda_k}
\]
then
$
\eta_{k,1},\eta_{k,2} = 1\pm \sqrt{1-\lambda_k} 
$,
which means the modulus of the eigenvalues of $I - A$ are $\sqrt{1-\lambda_k}$, with smallest nonzero value being $\sqrt{\rho}$. 
So,  since 
\[
\vx^{(k)} = [0,...,0,I] \vz^{(k)}, \quad
 \vr^{(k)} = [0,...,0,I](I-A)A^k \vz^{(0)},
 \]
 then

 \begin{eqnarray*}
     \|\vr^{(k)}\| &\leq&   \tau^2\rho(1-\sqrt{\rho})^{k}\|\vx^{(0)}\| + O((1-\sqrt{\rho})^{k-1}) \\
     &=& \tau^2\rho \exp(-k\sqrt{\rho})\|\vx^{(0)}\|+O(\exp(-(k-1)\sqrt{\rho})).
 \end{eqnarray*}
This is an acceleration of the problem's dependence on its condition value $1/\rho$. 
 \end{proof}

Note that in the above analysis, the specific BDF constants of $\xi_1 = -1/3$, $\xi_2 = 4/3$ apply specifically when $\rho = 3/4$, and provide an acceleration from $O((3/4)^k)$ to $O((1/2)^k)$. For a wider range of $\rho \in (0,2/3)$, acceleration of a lesser amount of acceleration is also possible is also possible. 
It is also to find higher order accelerating schemes, by solving higher order polynomials.

\paragraph{Discussion.} In general, tuning multi-step weights indeed has the potention to offer improved convergence behavior; however, tuning $\xi_i$ is a tedious process, and may require access to problem parameters that are difficult to compute. For example, knowing $\rho$ for two random subspaces requires knowing the smallest eigenvalue of $\Pi_1\Pi_2$, which in computing (for example via power iteration) would require as much complexity as solving the original problem. Therefore, this section is largely presented to satisfy a theoretical curiosity, but we believe the original BDF weights are better choices in practice. 

\section{Conclusion}
\label{sec:conclusion}
The goal of this work is to investigate the use of approximate implicit discretizations of \eqref{eq:prox_flow} in badly conditioned non-smooth problem settings. In this work, we found that the higher-order approximate implicit discretization helps in many optimization problems. However, it is worth pointing out that it is also important to find an efficient manner of choosing the approximate methods. In our work, we focus on the BDF method, which works well in practice. It is worth also investigating other types of multi-step methods and a wider range of problems for which acceleration may be possible.

\clearpage
\bibliographystyle{plainnat}
\bibliography{refs.bib}

\end{document}